\documentclass[11pt]{article}
\usepackage{amsmath,amsthm,amssymb}
\usepackage[T1]{fontenc}
\usepackage[utf8]{inputenc}
\usepackage[dvipdf]{graphicx}
\usepackage{color}
\usepackage{epstopdf}
\usepackage{dsfont}
\usepackage{mathtools}
\usepackage{enumerate}
\usepackage{mathrsfs}
\usepackage[normalem]{ulem}
\usepackage[colorlinks=true,citecolor=red,linkcolor=db,urlcolor=blue,pdfstartview=FitH]{hyperref}
\usepackage[numbers]{natbib}
\allowdisplaybreaks
\definecolor{db}{RGB}{0, 0, 130}
\definecolor{rp}{rgb}{0.25, 0, 0.75}
\definecolor{dg}{rgb}{0, 0.6, 0}

\newtheorem{theorem}{Theorem}[section]

\newtheorem{definition}{Definition}[section]

\newtheorem{proposition}[definition]{Proposition}

\newtheorem{assumption}[definition]{Assumption}
\newtheorem{lemma}[definition]{Lemma}

\newtheorem{remark}{Remark}[section]
\def\R{\mathbb{R}}

\def\E{\mathbb{E}}
\def\N{\mathbb{N}}

\def\S{\mathbb{S}}

\def\F{\mathbb{F}}
\def\P{\mathbb{P}}
\def\G{\mathbb{G}}

\def\taub{\pmb{\tau}}
\def\x{\times}
\def\o{\circ}
\def\Om{\Omega}
\def\om{\omega}
\def\Fc{\mathcal{F}}
\def\Xo{\overline{X}}

\def\eps{\varepsilon}
\def\Xh{\widehat{X}}

\def\Tc{\mathcal{T}}

\def\Cc{\mathcal{C}}
\def\Lh{\widehat{L}}

\def\mut{\widetilde{\mu}}

\def\Bc{\mathcal{B}}
\def\Gc{\mathcal{G}}
\def\Lc{\mathcal{L}}

\def\Wc{\mathcal{W}}
\def\Pc{\mathcal{P}}

\def\x {\times}

\def\Xt{\widetilde X}

\def\muh{\widehat{\mu}}
\def\muo{\overline{\mu}}

\def\taut{\tilde \tau}

\def\Gh{\widehat{G}}

\def \proof{{\noindent \bf Proof }}
\def \endproof{\hbox{ }\hfill$\Box$}

\title{On the limit theory of mean field optimal stopping with non-Markov dynamics and  common noise}
\author{
Xihao He
\footnote{Department of Mathematics, Univerisity of Michigan. hexihao@umich.edu.}
}

\date{\today}

\begin{document}

\maketitle

\begin{abstract}
    This paper focuses on a mean-field optimal stopping problem with non-Markov dynamics and common noise, inspired by Talbi, Touzi, and Zhang \cite{TalbiTouziZhang1,TalbiTouziZhang3}.
    The goal is to establish the limit theory and demonstrate the equivalence of the value functions between weak and strong formulations. The difference between the strong and weak formulations lies in the source of randomness determining the stopping time on a canonical space.
    In the strong formulation, the randomness of the stopping time originates from Brownian motions.
    In contrast, this may not necessarily be the case in the weak formulation. Additionally, a $(H)$-Hypothesis-type condition is introduced to guarantee the equivalence of the value functions.
    The limit theory encompasses the convergence of the value functions and solutions of the large population optimal stopping problem towards those of the mean-field limit, and it shows that every solution of the mean field optimal stopping problem can be approximated by solutions of the large population optimal stopping problem.
\end{abstract}


\vspace{0.5em}

\noindent\textbf{MSC2010 subject classification:} 60G40, 60G44, 60G07, 93E15

\section{Introduction}

    We study a mean field optimal stopping problem with common noise, which can be described as follows:
    Let $T > 0$ be a finite horizon.
    The stopped state process $X^\tau$ is determined through a stopping time $\tau$ and a McKean-Vlasov stochastic differential equation(SDE)
    \begin{equation}\label{eq:intro_dynamics}
        X^\tau_t
        ~ = ~
        X_0 + \int_{0}^{t \wedge \tau}b(s, X^\tau_{s\wedge \cdot}, \mu_s) dt
        + \int_{0}^{t \wedge \tau}\sigma(s, X^\tau_{s\wedge \cdot}, \mu_s) dW_s
        + \int_{0}^{t \wedge \tau}\sigma_0(s, X^\tau_{s\wedge \cdot}, \mu_s) dB_s,
    \end{equation}
    where $W$, $B$ are two independent Brownion motions,
    for each $t \in [0,T]$,
    $\mu_t$ denotes the conditional law of $(X^\tau_{t\wedge \cdot}, W, \tau \wedge t)$ given the common noise $B$,
    which is also written as $\Lc(X^\tau_{t\wedge \cdot}, W, \tau \wedge t|B)$.
    Then the optimal stopping problem is formulated as
    \begin{equation*}
        \sup_{\tau}
        \E\bigg[
            \int_{0}^{T \wedge \tau}f(s,X^\tau_{s\wedge \cdot}, \mu_s) ds
            + g(\tau, X^\tau_{\cdot}, \mu_T)\bigg].
    \end{equation*}

    The McKean-Vlasov control problem has recently gained attention due to its close proximity to mean field games, as introduced in the pioneering work of Lasry and Lions \cite{LarsyLions} and Huang, Caines, and Malhamé \cite{HuangCainesMalhame}.
    For a more thorough discussion about the similarities and differences between these two theories, we refer to the work of Carmona, Delarue, and Lachapelle \cite{CarmonaDelarueLachapelle}.
    Controlled McKean-Vlasov dynamics are generally investigated by two methods: the Pontryagin maximum principle and the dynamic programming principle (DPP).
    In particular, we refer to the work of Carmona and Delarue \cite{CarmonaDelarue2015} for a comprehensive analysis of the former approach, and to Djete, Possamaï, and Tan \cite{DjetePossamaiTan} for a general DPP using abstract measurable selection arguments.

    This problem is typically justified by considering a large population optimal stopping problem that includes numerous dynamics.
    More specifically, for each positive integer $N$, we consider $N$ state processes $X^N = (X^{1,N,\taub^{N}}, \cdots, X^{N, N,\taub^{N}})$ that interact through their empirical measures. These are described by the following SDEs:
    for $i = 1, \cdots, N$, $t \in [0,T]$,
    \begin{align*}
        &
        \begin{split}
        X^{i,N,\taub^{N}}_t
        ~ = ~
        X^i_0 ~ + ~ \int_{0}^{t \wedge \tau^{i,N}}b\Big(s, X^{i,N,\taub^{N}}_{s\wedge \cdot}, \mu^{N,\taub^{N}}_s\Big) dt
        ~ & + ~ \int_{0}^{t \wedge \tau^{i,N}}\sigma\Big(s, X^{i,N,\taub^{N}}_{s\wedge \cdot}, \mu^{N,\taub^{N}}_s\Big) dW^i_s
        \\~ & + ~ \int_{0}^{t \wedge \tau^{i,N}}\sigma_0\Big(s, X^{i,N,\taub^{N}}_{s\wedge \cdot}, \mu^{N,\taub^{N}}_s\Big) dB_s,
        \end{split}
        \\
        & ~~~
        \mu^{N,\taub^{N}}_t
        ~ := ~
        \frac{1}{N}\sum_{i = 1}^{N}
        \delta_{\big(X^{i,N, \taub^N}_{t \wedge \cdot}, W^i, \tau^i \wedge t\big)},
    \end{align*}
    where the $N$ stopping times $\taub^{N}:= (\tau^{1,N}, \cdots, \tau^{N,N})$ are fixed or chosen, $W^1, \cdots, W^N$, $B$ are a family of independent Brownian motions and $X^1_0, \cdots, X^N_0$ are independent initial values.
    Then we consider the $N$-player optimal stopping problem
    \begin{equation*}
        \sup_{\taub^{N}}
        \frac{1}{N}\sum_{i = 1}^{N}
            \E
                \bigg[
                    \int_{0}^{T \wedge \tau^{i,N}}f\Big(s,X^{i,N,\taub^N}_{s\wedge \cdot}, \mu^{N,\taub^{N}}_s\Big) ds
                    + g\Big(\tau^{i,N}, X^{i,N,\taub^N}_{s\wedge \cdot}, \mu^{N,\taub^{N}}_T\Big)
                \bigg].
    \end{equation*}
    We refer to Kobylanski, Quenez and Rouy-Mironescu \cite{KobylanskiQuenezRouy} for a detailed discussion of multi optimal stopping problem in a more general setting.

    In this paper, our objective is to investigate the limit theory for the mean field optimal stopping problem in the presence of common noise.
    Without consideration of the stopping time, such limit theory is generally referred to as propagation of chaos. This concept was initially studied by Kac \cite{Kac} and McKean Jr. \cite{McKean}. For more insightful discussions on this topic, Sznitman's lecture notes \cite{Sznitman} are highly recommended.

    In the controlled case, the convergence of the value function has been studied by several researchers. Lacker \cite{Lacker} examined it in a general setting without common noise, while Djete, Possamaï, and Tan \cite{DjetePossamaiTan} extended the study to include common noise and the joint law of the state process and control in a more comprehensive setting. Both of these studies \cite{Lacker, DjetePossamaiTan} introduced suitable martingale problems.

    The convergence rate of the value function has been further investigated using PDE techniques. Cecchin \cite{Cecchin} examined it in the finite state space, while Baryaktar, Cecchin, and Chakraborty \cite{BaryaktarCecchinChakraborty} studied the continuous state space with regime switching in the state dynamics. Germain, Pham, and Warin \cite{GermainPhamWarin} conducted their study under the assumption that the limit value is smooth. Cardaliaguet, Daudin, Jackson, and Souganidis \cite{Cardaliaguet1}, and Cardaliaguet and Souganidis \cite{Cardaliaguet2} carried out their investigations under a decoupling assumption on the Hamiltonian.

    In the stopped case, the literature is quite limited. Talibi, Touzi, and Zhang \cite{TalbiTouziZhang1, TalbiTouziZhang2, TalbiTouziZhang3} established well-posedness for an obstacle equation in Wasserstein distance from the mean field optimal stopping problem, as well as the corresponding Dynamic Programming Principle (DPP) and limit theory for the problem.

    The primary goal of this paper is to demonstrate the convergence result of the value function and optimizers of the $N$-player optimal stopping problem towards those of the mean field optimal stopping problem with common noise as $N$ tends to $+\infty$. We also aim to establish the equivalence between the value functions of its strong and weak formulations

    We first formulate both a strong and a weak formulation of the mean field optimal stopping problem, inspired by Djete, Possamaï, and Tan \cite{DjetePossamaiTan}. Heuristically, the feature of its strong formulation is that there exists a measurable function $\phi$ such that $\tau = \phi(X_0, W, B)$. In contrast, the weak formulation does not necessarily require such a measurable function and instead introduces a $(H)$-Hypothesis-type condition, which is crucial for establishing our equivalence result of the value functions.

    For simplicity, both formulations are introduced in a canonical space, allowing the control term to be formulated as probability measures on this space. These are defined as strong and weak stopped rules, respectively. We refer to \cite{DjetePossamaiTan} for the equivalence between our strong formulation and the formulation with a fixed probability space equipped with two Brownian motions $W$, $B$, and their natural filtrations.
    To prove the equivalence between value functions, we demonstrate that each weak stopped rule can be approximated by a sequence of strong stopped rules.

    In the remainder of the paper, we demonstrate propagation of chaos for our stopped McKean-Vlasov dynamics, leading to an estimation between the value function of the stopped McKean-Vlasov dynamics and that of the $N$-player optimal stopping problem. By introducing a suitable martingale problem, we show the convergence of the $\eps_N$-solution to the $N$-player optimal stopping problem towards the corresponding stopped McKean-Vlasov dynamics. For the convergence of the martingale problem, we adopt a different proof from \cite{Lacker}, which is not valid with the presence of common noise.

    The paper is organized as follows. In Section 2, we introduce the mean field optimal stopping problem, as well as the large but finite population optimal stopping problem. In Section 3, we present the main results of the paper. Section 4 is devoted to proofs, where we divide it into three subsections for preliminary results and one subsection for the proof of the main results.

\section{Mean field optimal stopping problem: different formulations}

    Let the horizon $T < +\infty$.
    For any measurable space $(\Om,\Fc)$,
    let $\Pc(\Om)$ denote the space of all probability measures on $(\Om,\Fc)$.

    For any Polish space $(E,d)$,
    let $C([0,T],E)$ denote the space of all $E$-valued continuous functions on $[0,T]$ equipped with the uniform norm $\|\cdot\|$,
    $\Bc(E)$ denote the collection of all Borel sets on $E$,
    $C_b(E)$ denote the space of all $\R$-valued bounded continuous functions on $E$,
    $\Pc_p(E)$ denote the space of all probability measures $\rho$ on $E$ such that
    \begin{equation*}
        \int_E d(x,x_0)^p \rho(dx) < +\infty,
    \end{equation*}
    For any $\rho \in \Pc(E)$ and $\phi \in C_b(E)$, we define $\langle \phi,\rho \rangle := \int_E\phi(x)\rho(dx)$.

    For any $\P \in \Pc(\Om)$, $E$-valued random variable $Y$, and any sub $\sigma$ algebra $\Gc$ of $\Fc$, we define $\Lc^{\P}(Y) := \P \circ Y^{-1}$ and $\Lc^{\P}(Y|\Gc) := \P^{\Gc}_\om \circ Y^{-1}$, where $\{\P^{\Gc}_\om\}_{\om \in \Om}$ denote the conditional probability distributions of $\P$ knowing $\Gc$.

    For $p \ge 1$, the Wasserstein $p$-distance $\Wc_p$ between two probability measures $\rho_1$ and $\rho_2$ with $\rho_1, \rho_2 \in \Pc_p(E)$ is defined as
    \begin{equation*}
        \Wc_p(\rho_1,\rho_2)
        ~ := ~
        \bigg(
            \inf_{\rho \in \Gamma(\rho_1,\rho_2)}
            \int_{E \x E}d(e_1,e_2)^p \rho(de_1,de_2)
        \bigg)^{1/p},
    \end{equation*}
    where $\Gamma(\rho_1,\rho_2)$ is the collection of all probability measures such that $\rho(de,E) = \rho_1(de)$ and $\rho(E, de) = \rho_2(de)$.

    \vspace{0.5em}

    Now for any positive integer $m$, $m_1$ and $m_2$,
    let $\Wc^{m}$ denote $m$-dimensional Wiener measure,
    $\Cc^m := C([0,T],\R^m)$,
    and
    $\S^{m_1 \x m_2}$ the space of all $m_1 \x m_2$-dimensional matrices with real entries, equipped with the standard Euclidean norm $|\cdot|$.

    We will introduce a series of functions and their assumptions which are in force throughout the paper.
    Let $n,d \in \N_+$, $\ell \in \N$ be strictly positive integers and positive integer, respectively.
    The stopped diffusion process \eqref{eq:intro_dynamics} has the following coefficient functions:
    \begin{equation*}
        (b, \sigma, \sigma_0): [0,T] \x \Cc^n \x \Pc(\Cc^n \x \Cc^d \x [0,T]) \longrightarrow \R^n \x \S^{n \x d} \x \S^{n \x \ell},
    \end{equation*}
    and the objective function is composed of the running reward function $f$ and terminal reward function $g$ defined as:
    \begin{equation*}
        f: [0,T] \x \Cc^n \x \Pc(\Cc^n \x \Cc^d \x [0,T]) \longrightarrow \R
        ~\mbox{and}~
        g: [0,T] \x \Cc^n \x \Pc(\Cc^n \x \Cc^d \x [0,T]) \longrightarrow \R
    \end{equation*}

    \begin{assumption}\label{ass:coefficient_reward}
        The maps $(b,\sigma,\sigma_0,f,g)$ are Borel measurable,
        and for some $p \ge 2$:

        \noindent $\mathrm{(i)}$
        the maps $(b,\sigma,\sigma_0)$ are uniformly Lipschitz in $(x,m)$,
        i.e. there exists a constant $L > 0$
        such that for all for all $(t, x_1, x_2, m_1, m_2) \in [0,T] \x \Cc^n \x \Cc^n \x \Pc_p(\Cc^n \x \Cc^d \x [0,T]) \x \Pc_p(\Cc^n \x \Cc^d \x [0,T])$,
        \begin{equation*}
            |(b,\sigma,\sigma_0)(t, x_1, m_1) - (b,\sigma,\sigma_0)(t, x_2, m_2)|
            ~ \le ~
            L\big(\|x_1 - x_2\| + \Wc_p(m_1,m_2)\big),
        \end{equation*}

        \noindent $\mathrm{(ii)}$
        the function $g$ is bounded lower semi-continuous in $(t, x, m)$,
        and for all $t \in [0,T]$, the function $f$ is bounded, lower semi-continuous in $(x,m)$.
    \end{assumption}
    
    Consequently, one can obtain that there exists some constant $C >0$, such that for all $(t, x, m) \in [0,T] \x \Cc^n \x \Pc_p(\Cc^n \x \Cc^d \x [0,T])$,
    \begin{align*}
        |(b,\sigma,\sigma_0)(t, x, m)|^p
         ~ \le &~
        C\big(1 + \|x\|^p + \Wc^p_p(m,\delta_{(0,0,0)})\big)
        \\ ~ \le &~
        C\bigg(1 + \|x\|^p + \int_{\Cc^n \x \Cc^d \x [0,T]}(\|x\| + \|w\| + |t|)^p m(dx,dw,dt)\bigg)
        \\ ~ \le &~
        C\bigg(1 + \|x\|^p + \int_{\Cc^n \x \Cc^d \x [0,T]}(\|x\|^p + \|w\|^p) m(dx,dw,dt)\bigg).
    \end{align*}
    
    \vspace{0.5em}

    \subsection{Strong and weak formulations of the mean field optimal stopping problem on canonical space}
        Let us consider the canonical space
        $$
            \Om := \Cc^n \x \Cc^d \x \Cc^\ell \x \Pc_p(\Cc^n \x \Cc^d \x [0,T]) \x [0,T]
        $$
        and the canonical process $(X = (X_t)_{t\in [0,T]}, W =(W_t)_{t\in [0,T]}, B = (B_t)_{t\in [0,T]}, \tau, \mu)$ given by the natural projections, $(X, W, B): \Om \longrightarrow \Cc^{n} \x \Cc^d \x \Cc^\ell$, $\tau: \Om \longrightarrow [0,T]$ and $\mu: \Om \longrightarrow \Pc_p(\Cc^n \x \Cc^d \x [0,T])$, which are defined by
        $$
            (X_t, W_t, B_t)(\om) := (x_t,w_t,b_t),
            ~
            \tau(\om) := \theta,
            ~
            \mu(\om) := m,
            ~
            \mbox{for all}
            ~
            \om = (x, w, b, \theta, m) \in \Om.
        $$
        Each natural projection in the canonical space, as well as the combination of these projections, will generate different filtrations.
        We need to introduce the following notations which we shall use in the sequel. For the continuous processes $X$, $W$ and $B$, they generate filtrations in the natural way as follows
        \begin{equation*}
            \F^Y := (\Fc^Y_t)_{t \in [0,T]},
            ~
            \Fc^Y_t := \sigma(Y_s, s \in [0,t]),
            ~
            \mbox{for}~
            t \in [0,T],
            ~
            Y = X, W ~\mbox{or}~ B.
        \end{equation*}
        For the random time $\tau$, the filtration $\F^\tau$ generated by it is given as
        \begin{equation*}
            \F^\tau := (\Fc^\tau_t)_{t \in [0,T]},
            ~
            \Fc^\tau_t := \sigma(\tau \wedge s, s \in [0,t])
            ~
            \mbox{for}~
            t \in [0,T].
        \end{equation*}
        The filtration $\F^{\mu}$ generated by $\mu$ is defined as
        \begin{equation*}
            \F^{\mu} := (\Fc^{\mu}_t)_{t \in [0,T]},
            ~
            \Fc^{\mu}_t := \sigma({\mu} \o \pi^{-1}_t)
            ~
            \mbox{for}~
            t \in [0,T],
        \end{equation*}
        where the projection map $\pi_t$ is defined on $\Cc^n \x \Cc^d \x [0,T]$ by $\pi_t(x, w, \theta) = (x_{t\wedge \cdot}, w_{t\wedge \cdot}, \theta \wedge t)$.

        For any filtration generated by multiple projections on the canonical space $\Om$, for instance, $B$ and $\mu$, we may simply denote the filtration by $\F^{B,\mu} := \F^B \vee \F^\mu$.
        In particular, the full filtration $\F^{X,W,B,\tau,\mu}$ is hereinafter abbreviated as $\F$ for simplicity.
        In particular, we denote by $\G := (\Gc_t)_{t \in [0,T]}$, the filtration generated by $B$ and $\mu$, i.e., $\Gc_t := \Fc^{B,\mu}_t$, for all $t \in [0,T]$.
        
        For any filtration $\mathbb{H}$ on the canonical space $(\Om, \Fc_T)$, we write $\mathbb{H}_+$ for the right-continuous filtration $(\mathcal{H}_{t+})_{t \in [0,T]}$, where $\mathcal{H}_{t+} := \bigcap_{s > t}\mathcal{H}_s$ for $t \in [0,T)$.
        In particular, $\tau$ is a $\F_+$-stopping time.

        Given the aforementioned canonical space and its attributes, we can now introduce the weak stopped rule. This rule forms the admissible set for the mean field optimal stopping problem.
        \begin{definition}\label{def:strong_weak_stopping}
            Let $\nu \in \Pc_p(\R^n)$.
            A probability $\P$ on $(\Om,\Fc_T)$ is called a weak stopped rule associated with $\nu$ if
            \begin{enumerate}[(i)]
              \item $(B,W)$ is a $(\F, \P)$-Brownian motion, 
                            $\tau$ is a $\F_+$-stopping time.
              \item $(B,\mu)$ is independent of $(X_0, W)$.
              \item It holds $\P$-a.s. that $\Lc^\P(X_0) = \nu$ and, for all $t \in [0,T]$,
                    \begin{equation*}
                        X_t
                        ~ = ~
                        X_0 + \int_{0}^{t \wedge \tau}b(s, X_{s\wedge \cdot}, \mu_s) dt
                        + \int_{0}^{t \wedge \tau}\sigma(s, X_{s\wedge \cdot}, \mu_s) dW_s
                        + \int_{0}^{t \wedge \tau}\sigma_0(s, X_{s\wedge \cdot}, \mu_s) dB_s.
                    \end{equation*}
              \item It holds $\P$-a.s. that for all $t \in [0,T]$,
                    \begin{equation}\label{eq:consistency_weak}
                        \mu_t = \Lc^\P((X_{t\wedge \cdot}, W, \tau \wedge t)|\Gc_t).
                    \end{equation}
              \item It holds $\P$-a.s. that for all $t \in [0,T]$, and $D \in \Fc_t \vee \Fc^W_T$,
                    \begin{equation}\label{eq:H_hypothesis}
                        \E^\P[\mathds{1}_{D}|\Gc_T]
                        ~ = ~
                        \E^\P[\mathds{1}_{D}|\Gc_t].
                    \end{equation}
            \end{enumerate}
            A probability $\P$ on $(\Om,\Fc_T)$ is called a strong stopped rule associated with $\nu$ if it is a weak stopped rule and
            $\tau$ is a stopping time w.r.t. the $\P$-completion of $\F^{X_0,W,B}$.
        \end{definition}
        \begin{remark}
            $(i)$ For the term (iii), it is equivalent to state that it holds true $\mu_T$-a.e., $\P$-a.s.
    
            \noindent $(ii)$ For the term (v), it is usually called a $(H)$-hypothesis-type condition. 
            Technically speaking, the condition is used to obtain the equivalence between strong and weak value functions. 
            In particular, the properties \eqref{eq:consistency_weak_m} and \eqref{eq:consistency_m} in Lemma \ref{lemma:equivalence_approximation} are proved by this condition.
            These properties are later on used to prove Lemma \ref{lemma:equivalence_strong_weak_discrete}.
            Generally speaking, 
            this condition is to guarantee that there exists a sequence of $\{\mu^m\}_m$, which are conditional laws of $(X, W, \tau^m)$ w.r.t. $B$, that can converge in distribution to the $\mu$ in any weak stopped rule, where $\{\tau^m\}_m$ are a sequence of $\F^{X_0,W,B}$-stopping times.
            One may refer to Section 4.2 in \cite{ElKarouiTan2} for general intuition, and a detailed discussion.
            
            \noindent $(iii)$ For existence, uniqueness, and $L^p$ boundedness of the strong solutions to the Mckean-Vlasov SDE we consider in this definition, and those new Mckean-Vlasov SDEs we will consider in the following content, one may refer to Theorem 5.1.1 in \cite{StroockVaradhan} for general idea, or to Theorem A.3 in \cite{DjetePossamaiTan1} for the corresponding proof in control setting, which is almost the same. 
        \end{remark}
        We can now define the sets $\Pc_W(\nu)$ and $\Pc_S(\nu)$, which denote the collection of all weak stopped rules associated with $\nu$ and that of all strong stopped rules associated with $\nu$, respectively.
        With this notation established, we proceed to formally define the weak and strong McKean-Vlasov stopping time problem as follows:
        \begin{equation*}
            V_W(\nu)
            ~ := ~
            \sup_{\P \in \Pc_W(\nu)}
            J(\P),
            ~
            V_S(\nu)
            ~ := ~
            \sup_{\P \in \Pc_S(\nu)}
            J(\P)
        \end{equation*}
        with
        \begin{equation*}
            J(\P)
            ~ := ~
            \E^{\P}\bigg[
            \int_{0}^{T \wedge \tau}f(s,X_{s\wedge \cdot}, \mu_s) ds
            + g(\tau, X_{\cdot}, \mu_T)\bigg].
        \end{equation*}
        Next, we denote by $\Pc^*_W(\nu)$ the collection of all elements $\P^*$ in $\Pc_W(\nu)$ such that
        $
            V_W(\nu) ~ = ~ J(\P^*).
        $

    \subsection{A large population optimal stopping problem with common noise}\label{subsec:N_player}
        Let $(\Om^0,\Fc^0,\P^0)$ be an abstract probability space equipped with a sequence of independent $d$-dimensional Brownian motions $\{W^{i}\}_{i \in \N_+}$,
        a $\ell$-dimensional Brownian motion $B$ that is independent of $W^i$ for all $i \in \N_+$,
        a sequence of random variables $\{X^i_0\}_{i \in \N_+}$ independent of $\{B,W^1, \cdots, W^n, \cdots\}$, 
        and a sequence of probability measures $\{\nu_i\}_{i \in \N_+} \subset \Pc_p(\R^n)$.

        For each $N \in \N$, let $\F^N := \{\Fc^N_t\}_{t \in [0,T]}$ denote the augmented filtration generated by $\{X^i_0\}_{1 \le i \le N}$, $\{W^i_0\}_{1 \le i \le N}$ and $B$. Additionally, let $\Tc^N$ be the set of $\F^N$-stopping times taking values in $[0,T]$.

        Given any $\taub^{N}:= (\tau^{1,N}, \cdots, \tau^{N,N}) \in (\Tc^N)^N$, there exists a sequence of the continuous processes $\mathbf{X}^{N,\taub^{N}}:= (X^{1,N,\taub^{N}}, \cdots, X^{N,N,\taub^{N}})$, which are the unique strong solutions to the following a large population stopped SDEs:
        for $i = 1, \cdots, N$, $t \in [0,T]$,
        \begin{equation}\label{eq:N_SDE}
            \begin{split}
            X^{i,N,\taub^{N}}_t
            ~ = ~
            X^i_0 ~ + ~ \int_{0}^{t \wedge \tau^{i,N}}b(s, X^{i,N,\taub^{N}}_{s\wedge \cdot}, \mu^{N,\taub^{N}}_s) dt
            ~ & + ~ \int_{0}^{t \wedge \tau^{i,N}}\sigma(s, X^{i,N,\taub^{N}}_{s\wedge \cdot}, \mu^{N,\taub^{N}}_s) dW^i_s
            \\~ & + ~ \int_{0}^{t \wedge \tau^{i,N}}\sigma_0(s, X^{i,N,\taub^{N}}_{s\wedge \cdot}, \mu^{N,\taub^{N}}_s) dB_s,
            \end{split}
        \end{equation}
        respectively, where for each $t \in [0,T]$,
        \begin{equation}\label{eq:N_SDE_consistency}
            \mu^{N,\taub^{N}}_t
            ~ := ~
            \frac{1}{N}\sum_{i = 1}^{N}
            \delta_{\big(X^{i,N,\taub^{N}}_{t\wedge \cdot}, W^i, \tau^{i,N}\wedge t\big)}
            ~\mbox{for each}~
            t \in [0,T],
            ~\mbox{and}~
            \Lc^{\P^0} (X^i_0)
                ~ = ~
            \nu_i.
        \end{equation}
        The $N$-player optimal stopping problem is given by
        \begin{equation}
            V^N_S(\nu_1, \cdots, \nu_N)
            ~ := ~
            \sup_{\taub^{N} \in (\Tc^N)^N}
            J_N(\taub^{N})
        \end{equation}
        with for each $\taub^{N} \in (\Tc^N)^N$,
        \begin{equation}
            J_N(\taub^{N})
            :=  ~
            \frac{1}{N}\sum_{i = 1}^{N}
            \E^{\P^0}
                \bigg[
                    \int_{0}^{T \wedge \tau^{i,N}}f(s,X^{i,N,\taub^N}_{s\wedge \cdot}, \mu^{N,\taub^{N}}_s) ds
                    + g(\tau^{i,N}, X^{i,N,\taub^N}_{s\wedge \cdot}, \mu^{N,\taub^{N}}_T)
                \bigg].
        \end{equation}
        Hence, for each $\taub^{N} \in (\Tc^N)^N$
        we can construct a probability measure $\P_N$, defined as
        \begin{equation}\label{eq:P_N_construct}
            \P_N
            ~ := ~
            \frac{1}{N}\sum_{i = 1}^{N}
            \Lc^{\P^0}
            \Big(X^{i,N,\taub^N}, W^i, B, \frac{1}{N}\sum_{i = 1}^{N}\delta_{(X^{i,N,\taub^{N}}, W^i, \tau^{i,N})}, \tau^{i,N}\Big).
        \end{equation}
        and we may notice that
        $
            J_N(\taub^N) = J(\P_N).
        $

\section{Main results}

    \begin{theorem}\label{thm:equivalence_value_function}
        Let Assumption \ref{ass:coefficient_reward} holds true, $\nu \in \Pc_p(\R^n)$.
        Then the set $\Pc_S(\nu)$ is dense in the compact, convex set $\Pc_W(\nu)$ under the Wasserstein distance $\Wc_p$.
        Consequently, it holds that
        \begin{equation*}
            V_S(\nu) = V_W(\nu).
        \end{equation*}
    \end{theorem}

    \begin{theorem}\label{thm:main}
        Let Assumption \ref{ass:coefficient_reward} holds true, $\{\nu_i\}_{i \in \N_+} \subset \Pc_p(\R^n)$ be such that for some $\eps > 0$,
            $$
                \sup_{N \ge 1}\frac{1}{N}\sum_{i = 1}^{N}\int_{\R^n}|x|^{p + \eps}\nu_i(dx)
                ~ < ~
                +\infty.
            $$
        and let the function $g$ be continuous in $(t, x, m)$,
        and for all $t \in [0,T]$, the function $f$ be continuous in $(x,m)$.
        \begin{enumerate}[(i)]
          \item For any sequence of stopping times $\{\taub^N := (\tau^{1,N}, \cdots, \tau^{N,N}) \in (\Tc^N)^N\}_{N \in \N_+}$
              together with a sequence of positive numbers $(\eps_N)_{N \in \N_+} \subset \R_+$ with
                $
                    \lim_{N \to \infty}
                    \eps_N
                    ~ = ~
                    0,
                $
              such that
              \begin{equation*}
                J_N(\taub^{N}) \ge V^N_S(\nu_1, \cdots, \nu_N) - \eps_N,
                ~\mbox{for each}~
                N \in \N_+,
              \end{equation*}
              it holds that the sequence of probability measures $\{\P_N\}_{N \in \N_+}$, given by
                \begin{equation*}
                    \P_N
                    ~ := ~
                    \frac{1}{N}\sum_{i = 1}^{N}
                    \Lc^{\P^0}
                    \Big(X^{i,N,\taub^N}, W^i, B, \frac{1}{N}\sum_{i = 1}^{N}\delta_{(X^{i,N,\taub^{N}}, W^i, \tau^i)}, \tau^i\Big),
                \end{equation*}
               constructed as \eqref{eq:P_N_construct}, is relatively compact under $\Wc_p$ and every limit point of $\{\P_N\}_{N \in \N_+}$ belongs to $\Pc^*_W(\nu)$ for some $\nu \in \Pc_p(\R^n)$ with
                \begin{equation*}
                    \lim_{m \to \infty}\frac{1}{N_m}\sum_{i = 1}^{N_m}\nu_i ~ = ~\nu.
                \end{equation*}
          \item If in addition we assume that
                \begin{equation*}
                    \lim_{N \to \infty}\frac{1}{N}\sum_{i = 1}^{N}\nu_i ~ = ~\nu,
                    ~\mbox{for some}~
                    \nu \in \Pc_p(\R^n).
                \end{equation*}
                Then it holds that
                \begin{equation*}
                    \lim_{N \to \infty}V^N_S(\nu_1, \cdots, \nu_N)
                    ~ = ~
                    V_S(\nu).
                \end{equation*}
                Moreover, for any $\P^* \in \Pc^*_W(\nu)$, there exists a sequence of stopping times $\{\taub^{N}\}_{N \in \N_+}$, where $\taub^{N} \in (\Tc^N)^N$, for all $N \in \N_+$, and a sequence of positive numbers $\{\eps_N\}_{N \in \N_+} \subset \R_+$ with
                $
                    \lim_{N \to \infty}
                    \eps_N
                    ~ = ~
                    0,
                $
                such that
                \begin{equation*}
                    J_N(\taub^{N}) \ge V^N_S(\nu_1, \cdots, \nu_N) - \eps_N,
                    ~\mbox{for each}~
                    N \in \N_+,
                \end{equation*}
                and
                \begin{equation*}
                    \lim_{N \to \infty}\Wc_p(\P_N, \P^*) ~ = ~ 0.
                \end{equation*}
          \item It holds that
                \begin{equation*}
                    \lim_{N \to \infty}
                    \bigg|V^N_S(\nu_1, \cdots, \nu_N) - V_S\Big(\frac{1}{N}\sum_{i = 1}^{N}\nu_i\Big)\bigg|
                    ~ = ~ 0.
                \end{equation*}
        \end{enumerate}
    \end{theorem}

    \begin{remark}
        $\mathrm{(i)}$
        When $\ell = 0$, all of our problems and arguments reduce to the context without common noise. However, in the case where $\ell \neq 0$ and $\sigma_0 = 0$, the situation differs from the scenario without common noise. This is because the Brownian motion $B$ can be viewed as an external noise, and it may not be independent from $\tau$ in the weak formulation. Then the conditional distribution $\mu$ in Definition \ref{def:strong_weak_stopping} may not be reduced to a deterministic distribution. 
        This new problem is also a problem with common noise, but a very special one, i.e. the common noise $B$ doesn't appear in the dynamics directly (no $d B_t$ term), but the strategy (stopping time) can rely on the common noise. In other words, in a large population optimal stopping problem, the admissible stopping strategies of all players can observe some common noise together, and this common noise can only affect the dynamics through the stopping strategies taken.
        
        \noindent $\mathrm{(ii)}$
        Unlike the controlled case, there is no relaxed formulation for the mean field optimal stopping problem. The primary reason is that the weak limit of a sequence of stopping times will always be a randomized stopping time, which aligns with the weak formulation.
        
        \noindent $\mathrm{(iii)}$
        Inspired by \cite{TalbiTouziZhang3}, we can view a stopping time as a control with nonincreasing paths and taking values in $\{0,1\}$. In other words, it's a control with excellent path regularity. This insight led me to realize that the techniques in \cite{DjetePossamaiTan} can be used to solve the mean field optimal stopping problem with non-Markov dynamics and common noise.

        \noindent $\mathrm{(iv)}$
        Compared to the results in \cite{TalbiTouziZhang3} without common noise, the results in this paper are almost the same but required weaker regularity on the coefficients.
        For assumptions, they require additionally continuity in time of the drift and volatility $(b,\sigma)$, differentiability in space and probability measure variables of $\sigma$, and that $(b,\sigma, f, g)$ can be extended to $\Pc_1$ space with corresponding continuity.
        For main results, they consider the value functions starting from different time, and obtained the convergence of the value functions in the $N$-population problem involving different initial times for different numbers of population.
        For the method, they consider the corresponding obstacle problem and use comparison principle to prove the convergence, while this paper adopts purely probabilistic approach, mainly compactness of admissible sets, and equivalence of value functions between weak and strong formulations.

        \noindent $\mathrm{(v)}$
        The convergence rate of the limit theory is usually proved by PDE methods. The obstacle problem, corresponding to the mean field optimal stopping problem we consider in this paper, is a second order PDE on Wasserstein space, which is infinite dimensional.
        Its comparison principle, if proved, can be used to prove the convergence rate of our problem.
    \end{remark}

\section{Proofs}

\subsection{Equivalence between different formulations}
    In this subsection, we present technical proofs demonstrating the equivalence between different formulations of the mean field optimal stopping problem. The proof process is organized into three steps, where the main idea of the proof is inspired by \cite{DjetePossamaiTan}.
    In the first step, we prove that any weak stopped rule can be approximated by a sequence of weak stopped rules with finite state.
    In the second step, an equivalence between the weak and strong stopped rule with finite state is proved.
    In the last step, we show that any limit of strong stopped rules in the topology of weak convergence is a weak stopped rule.
    Furthermore, a weak (or strong) stopped rule with finite state generally refers to a weak (or strong) stopped rule such that the corresponding stopping time takes values in a finite state. The exact differences will become more apparent as we delve into the details.

    In our first step, recall that a weak stopped rule $\P \in \Pc_W(\nu)$ is defined in the filtered canonical space $(\Om, \Fc_T, \F)$.
    Then in probability space $(\Om, \F, \P)$ for any sequence of partitions $\pi_m: (t_i^m)_{i = 0}^m$ of $[0,T]$ with $0 = t_0^m < t_1^m < \cdots < t_m^m = T$ such that
    \begin{equation*}
        \lim_{m \to \infty}
            |\pi_m|
        ~ = ~
            0,
    \end{equation*}
    where $\sup_{1 \le i \le m}|t_{i}^m - t_{i - 1}^m|$.
    due to some technical reasons, we introduce a sequence of Brownian motions $(B^m, W^m)$  to prove Lemma \ref{lemma:equivalence_strong_weak_discrete} (the detailed reason will be explained in Remark \ref{rem:equivalence_strong_weak_discrete}), starting from time $t_1^m$ by
    \begin{equation*}
        B^m := B_{\cdot \vee t_1^m} - B_{t_1^m},
        ~
        W^m := W_{\cdot \vee t_1^m} - W_{t_1^m},
    \end{equation*}
    and filtrations $\F^m = (\Fc^m_t)_{t \in [0,T]}$, $\G^m = (\Gc^m_t)_{t \in [0,T]}$ by
    \begin{equation*}
        \Fc^m_t
            :=
        \sigma\big(X_{t \wedge \cdot}, W^m_{t \wedge \cdot}, B^m_{t \wedge \cdot}, \mu_t, \tau \wedge t\big),
        ~
        \Gc^m_t
            :=
        \sigma(B^m_{t \wedge \cdot}, \mu_t).
    \end{equation*} 

    \begin{lemma}\label{lemma:equivalence_approximation}
        Let Assumption \ref{ass:coefficient_reward} holds true.
        For any $\P \in \Pc_W(\nu)$,
        there exists a sequence of $\F$-stopping times $(\tau^m)_{m \in \N_+}$
        with
        \begin{equation*}
            \tau^m \wedge t = \tau^m \wedge t^m_i,
            ~\mbox{for}~
            t \in [t^m_i, t^m_{i + 1}),
            ~
            i = 0, 1, \cdots, m - 1,
        \end{equation*}
        such that
        \begin{equation}\label{eq:lemma_approximation}
            \lim_{m \to \infty}\sup_{\om \in \Om}|\tau^m(\om) - \tau(\om)| =
            0,
            ~
            \lim_{m \to \infty} \E^\P\bigg[\sup_{s \in [0,T]}|X^m_s - X_s|^p\bigg]
            =
            0,
        \end{equation}
        where $X^m$ is the unique strong solution of the following McKean-Vlasov SDE
        \begin{equation}\label{eq:SDE_weak_discrete}
            \begin{split}
            X^m_t
            ~ = ~
            X_0 + \int_{t^m_1}^{\tau^m \wedge t\vee t^m_1}b(s, X^m_{s\wedge \cdot}, \mu^m_s) dt
            & ~ + \int_{t^m_1}^{\tau^m \wedge t\vee t^m_1}\sigma(s, X^m_{s\wedge \cdot}, \mu^m_s) dW^m_s
            \\ & ~ + \int_{t^m_1}^{\tau^m \wedge t \vee t^m_1}\sigma_0(s, X_{s\wedge \cdot}, \mu^m_s) dB^m_s,
            ~
            t \in [0,T],
            \end{split}
        \end{equation}
        with
        \begin{equation*}
            \mu^m_t
            ~ := ~
            \Lc^\P((X^m_{t \wedge \cdot}, W^m, \tau^m \wedge t)|\Gc^m_t),
            ~
            t\in [0,T].
        \end{equation*}
        Then it follows that
        \begin{equation}\label{eq:convergence_law_m}
            \lim_{m \to \infty}
            \Wc_p\big(\Lc^\P(X^m, W^m, B^m, \mu^m_T, \tau^m),
                  \Lc^\P(X, W, B, \mu_T, \tau)\big)
            ~ = ~
            0.
        \end{equation}
        Moreover,
        it holds that that $(X_0, W^m)$ is $\P$-independent of $(B^m, \mu^m_T)$ and,
        that 
        \begin{align}\label{eq:H_hypothesis_m}
            & ~
            \E^\P[\mathds{1}_D|\Gc^m_T]
            ~ = ~
            \E^\P[\mathds{1}_D|\Gc^m_t],
            ~\mbox{for all}~
            t \in [0,T],
            D \in \Fc^m_t \vee \sigma(W^m),
        \\
        \label{eq:consistency_weak_m}
            & ~
            \mu_t
            =
            \Lc^\P((X_{t \wedge \cdot}, W, \tau \wedge t)|\Gc^m_t)
            =
            \Lc^\P((X_{t \wedge \cdot}, W, \tau \wedge t)|\Gc^m_T),
            ~\mbox{for all}~
            t \in [0,T],
            \\ \label{eq:consistency_m}
            & ~
            \mu^m_t
            = \Lc^\P((X^m_{t \wedge \cdot}, W^m, \tau^m \wedge t)|B^m_{t \wedge \cdot}, \mu^m_t)
            =
            \Lc^\P((X^m_{t \wedge \cdot}, W^m, \tau^m \wedge t)|B^m, \mu^m_T),
            ~\mbox{for all}~
            t \in [0,T].
        \end{align}
    \end{lemma}

    \begin{proof}
        Let us define a sequence of stopping times $(\tau^m)_{m \in \N_+}$ directly by
        \begin{equation*}
            \tau^m
            ~ := ~
            \sum_{i = 0}^{m - 1}t^m_i \mathds{1}_{[t^m_i,t^m_{i + 1})}( (\tau + 1/m) \wedge T),
            ~
            m \in \N_+.
        \end{equation*}
        In the rest part of the proof, we verify that $(\tau^m)_{m \in \N_+}$ is our desired sequence of stopping times in the lemma.

        First, it is obvious to notice that
        \begin{equation*}
            \tau^m \wedge t = \tau^m \wedge t^m_i,
            ~\mbox{for}~
            t \in [t^m_i, t^m_{i + 1}),
            ~
            i = 0, 1, \cdots, m - 1,
        \end{equation*}
        and
        \begin{equation*}
            \lim_{m \to \infty}
                \sup_{\om \in \Om}|\tau^m(\om) - \tau(\om)|
            ~ \le ~
            \lim_{m \to \infty}
                \sup_{0 \le i \le m - 1}|t^m_{i + 1} - t^m_i| + \frac{1}{m}
            ~ = ~
            0.
        \end{equation*}
        Next, we will prove the convergence of $X^m$ to $X$ in the sense of
        \begin{equation*}
             \lim_{m \to \infty} \E^\P\bigg[\sup_{s \in [0,T]}|X^m_s - X_s|^p\bigg]
            =
            0,
        \end{equation*}
        by standard arguments.
        In the following part of the proof, we introduce the function $I^\theta_t := \mathds{1}_{[0,\theta]}(t)$,
        for all $\theta, t \in [0,T]$.
        There exists a constant $C > 0$, which may vary from line to line, such that
        \begin{align*}
            & ~
                \E^\P\bigg[\sup_{s \in [0,T]}|X^m_s - X_s|^p\bigg]
            \\ ~ \le & ~
                \E^\P\bigg[\sup_{s \in [0,t^m_1]}|X^m_s - X_s|^p
                      +  \sup_{s \in [t^m_1,T]}|X^m_s - X_s|^p\bigg]
            \\ ~ \le & ~
                C\E^\P\bigg[\sup_{s \in [0,t^m_1]}|X_0 - X_s|^p
                      +  \int_{t^m_1}^{T}|
                                (b, \sigma, \sigma_0)(s, X_{s\wedge\cdot}, \mu_s)I^\tau_s
                                - (b, \sigma, \sigma_0)(s, X^m_{s\wedge\cdot}, \mu^m_s)I^{\tau^m}_s |^p ds\bigg]
            \\ ~ \le & ~
                C\E^\P\bigg[\sup_{s \in [0,t^m_1]}|X_0 - X_s|^p
                + \int_{t^m_1}^{T}\sup_{t^m_1 \le r \le s}|X^m_r - X_r|^p + \Wc^p_p(\mu_s, \mu^m_s) ds
                \\ & ~~~~~~~~~~~~~~~~~~~~~~~~~~~~~~~~
                + \int_{t^m_1}^{T}
                    |(b, \sigma, \sigma_0)(s, X_{s\wedge\cdot}, \mu_s)|^p|I^\tau_s - I^{\tau^m}_s| ds\bigg]
            \\ ~ \le & ~
                C\E^\P\bigg[\sup_{s \in [0,t^m_1]}|X_0 - X_s|^p
                + \sup_{0 \le s \le T}|W^m_s - W_s|^p 
                + \int_{t^m_1}^{T} \sup_{0 \le r \le s}|X^m_r - X_r|^p ds  
                \\& \quad\quad\qquad\quad
                + \bigg(1 + \sup_{s \in [0,t^m_1]}|X_s|^p + 
                \int_{\Cc^n \x \Cc^d \x [0,T]}\big(\|x\|^p + \|w\|^p \big) \mu_T(dx,dw,dt)\bigg)|\tau^m - \tau|\bigg]
            \\ ~ \le & ~
                C\E^\P\bigg[\sup_{s \in [0,t^m_1]}|X_0 - X_s|^p
                + \sup_{0 \le s \le T}|W^m_s - W_s|^p 
                + \int_{t^m_1}^{T} \sup_{0 \le r \le s}|X^m_r - X_r|^p ds  
                \\& \qquad \qquad\qquad\qquad\qquad\quad
                + \bigg(1 + \sup_{s \in [0,t^m_1]}|X_s|^p + 
                \E[\|X\|^p + \|W\|^p|\Gc_T]\bigg)(|\pi_m| + \frac{1}{m})\bigg]
            \\ ~ = & ~
                C\E^\P\bigg[\sup_{s \in [0,t^m_1]}|X_0 - X_s|^p
                + \sup_{0 \le s \le T}|W^m_s - W_s|^p 
                + \int_{t^m_1}^{T} \sup_{0 \le r \le s}|X^m_r - X_r|^p ds  
                \\& \qquad \qquad\qquad\qquad\qquad\quad
                + \bigg(1 + \E\bigg[\sup_{s \in [0,t^m_1]}|X_s|^p\bigg] + 
                \E[\|X\|^p + \|W\|^p]\bigg)(|\pi_m| + \frac{1}{m})\bigg]
            \\ ~ \le & ~
                C\E^\P\bigg[\sup_{s \in [0,t^m_1]}|X_0 - X_s|^p
                + (|\pi_m| + \frac{1}{m}) + \int_{0}^{T} \sup_{0 \le r \le s}|X^m_r - X_r|^p ds\bigg],
        \end{align*}
        where
        the second inequality follows by Burkholder-Davis-Gundy inequality, Hölder's inequality, 
        the third one by Lipschitz propetry of $(b, \sigma, \sigma_0)$,
         the fourth one by the definition of Wasserstein distance, $\mu_t = \Lc^\P(X_{t\wedge \cdot}, W, \tau\wedge t| \Gc^m_t)$ $\P$-a.s. for all $t \in [0,T]$, the Lipschitz of $(b, \sigma, \sigma_0)$,
         and the last one by $L^p$ boundedness of the $X$ under $\P$.

        Since the above estimation still holds true if we replace $T$ with any time $t \in [t^m_1,T]$, one has by Grownwall's lemma that
        \begin{equation*}
            \E^\P\bigg[\sup_{s \in [0,T]}|X^m_s - X_s|^p\bigg]
            ~ \le ~
            C\E^\P\bigg[\sup_{s \in [0,t^m_1]}|X_0 - X_s|^p
                +  (|\pi_m| + \frac{1}{m})\bigg]e^{CT}.
        \end{equation*}
        Then we can conclude our proof of \eqref{eq:convergence_law_m}, since it follows that
        \begin{align*}
            & ~
            \Wc^p_p\big(\Lc^\P(X^m, W^m, B^m, \mu^m_T, \tau^m),
                  \Lc^\P(X, W, B, \mu_T, \tau)\big)
            \\ ~ \le & ~
            2\E^{\P}\bigg[\sup_{s \in [0,T]}|X^m_s - X_s|^p + \sup_{s \in [0,T]}|W^m_s - W_s|^p + \sup_{s \in [0,T]}|B^m_s - B_s|^p + |\tau^m - \tau|^p\bigg].
        \end{align*}

        \vspace{0.5em}

        As for the second part, we first verify the equalities in \eqref{eq:consistency_weak_m}, \eqref{eq:H_hypothesis_m} and then \eqref{eq:consistency_m}.

        As for \eqref{eq:consistency_weak_m}, for any $\phi \in C_b(\Cc^n \x \Cc^d \x [0,T])$ and $\varphi \in C_b(\Cc^\ell \x \Pc_p(\Cc^n \x \Cc^d \x [0,T]))$, one has that, for each $t \in [0,T]$
        \begin{align*}
            ~ 
            \E^\P[
            \langle \phi,\mu_t \rangle
            \varphi(B,\mu_t)]
             = ~ &
            \E^\P[\E^\P[
            \langle \phi,\mu_t \rangle
            \varphi(B,\mu_t)|\Gc^m_t]]
             = ~ 
            \E^\P[
            \langle \phi,\mu_t \rangle
            \E^\P[\varphi(B,\mu_t)|\Gc^m_t]]
            \\ = ~ &
            \E^\P[
            \E^\P[\phi(X_{t \wedge \cdot}, W, \tau \wedge t)|\Gc_t]
            \E^\P[\varphi(B,\mu_t)|\Gc^m_t]
            ]
            \\ = ~ & 
            \E^\P[
            \phi(X_{t \wedge \cdot}, W, \tau \wedge t)
            \E^\P[\varphi(B,\mu_t)|\Gc^m_t]
            ]
            \\ = ~ &
            \E^\P[
            \E^\P[\phi(X_{t \wedge \cdot}, W, \tau \wedge t)|\Gc^m_t]
            \varphi(B,\mu_t)].
        \end{align*}
        Since the bounded continuous functions $\phi$, $\varphi$ are arbitrary and that the equality \eqref{eq:H_hypothesis} holds true, the equalities in \eqref{eq:consistency_weak_m} follow.

        For \eqref{eq:H_hypothesis_m}, it is direct to observe that $(X_0, W^m)$ is $\P$-independent of $(B^m, \mu^m_T)$, and \eqref{eq:H_hypothesis_m} follows by the argument below.
        For any $D \in \Fc^m_t \vee \sigma(W^m)$, $A \in \Gc^m_T$,
        \begin{align*}
            \E^\P[
            \mathds{1}_D\mathds{1}_A
            ]
            ~ = ~ 
            \E^\P[
                \E^\P[\mathds{1}_D|\Gc_T]\mathds{1}_A
                ]
            ~ = ~
            \E^\P[
                \E^\P[\mathds{1}_D|\Gc_t]\mathds{1}_A
                ]
            ~ = ~
            \E^\P[
                \E^\P[\mathds{1}_D|\Gc^m_t]\mathds{1}_A
                ],
        \end{align*}
        where the last equality holds by a similar argument for \eqref{eq:consistency_weak_m}.

        As for \eqref{eq:consistency_m}, by \eqref{eq:H_hypothesis_m}, it holds that, for each $t \in [0,T]$,
        \begin{equation*}
            \mu^m_t
            = \Lc^\P((X^m_{t \wedge \cdot}, W^m, \tau^m \wedge t)|\Gc^m_T).
        \end{equation*}
        Then the \eqref{eq:consistency_m} hold by a similar argument for \eqref{eq:consistency_weak_m}.
    \end{proof}

    Then in our second step, we prove the equivalence between weak stopped rule, which is rigorously defined in the statement of Lemma \ref{lemma:equivalence_approximation} and strong stopped rule with finite state.
    \begin{lemma}\label{lemma:equivalence_strong_weak_discrete}
        Let Assumption \ref{ass:coefficient_reward} holds true.
        For any $\P \in \Pc_W(\nu)$, any partition $\pi^m$ of $[0,T]$,
        and $\ell > 0$, there exists a $\F^{X_0, W, B}$-stopping time $\taut^m$
        such that
        \begin{equation*}
            \Lc^\P(\Xt^m, W^m, B^m, \mut^m, \taut^m)
            ~ = ~
            \Lc^\P(X^m, W^m, B^m, \mu^m, \tau^m),
        \end{equation*}
        where $\Xt^m$ is the unique solution of the following McKean-Vlasov SDE
        \begin{equation}\label{eq:SDE_strong_discrete}
            \begin{split}
            \Xt^m_t
            ~ = ~
            X_0 + \int_{t^m_1}^{t \vee t^m_1}b(s, \Xt^m_{s\wedge \cdot}, \mut^m_s)I^{\taut^m}_s dt
            & ~ + \int_{t^m_1}^{t \vee t^m_1}\sigma(s, \Xt^m_{s\wedge \cdot}, \mut^m_s)I^{\taut^m}_s dW^m_s
            \\ & ~ + \int_{t^m_1}^{t \vee t^m_1}\sigma_0(s, \Xt_{s\wedge \cdot}, \mut^m_s)I^{\taut^m}_s dB^m_s,
            ~
            t \in [0,T],
            \end{split}
        \end{equation}
        with for each $t \in [0,T]$, $\mut^m_t$ is defined by
        \begin{equation}\label{eq:consistency_m_strong}
            \mut^m_t
            ~ := ~
            \Lc^\P((\Xt^m_{t \wedge \cdot}, W^m, \taut^m \wedge t)|B),
            ~
            t\in [0,T].
        \end{equation}
    \end{lemma}
    \begin{remark}\label{rem:equivalence_strong_weak_discrete}
        Without introducing the new Brownian motions $(B^m, W^m)$, we have to replace every $(B^m,W^m)$ in 
        Lemma \ref{lemma:equivalence_strong_weak_discrete} with $(B,W)$.
        However, in this case, the corresponding Lemma \ref{lemma:equivalence_strong_weak_discrete} is not true. 
        The reason is that, in the following proof we use Lemma 4.7 in El Karoui and Tan \cite{ElKarouiTan2}, which required extra randomness $U$, which is independent from $(X_0, W^m, B^m)$, in $\taut^m$. In this case if $(W^m, B^m)$ is replaced by $(W,B)$, $\taut^m$ is not a $\F^{X_0,W,B}$-stopping time. 
        With introducing $(W^m, B^m)$, our strategy is take $U$ to be a function of $(W - W^m, B - B^m)$.
    \end{remark}

    \begin{proof}
        \textbf{Step.1}
        We first construct $\mut^m$, and $\taut^m$ at each discrete time points $\{t^m_0, \cdots, t^m_m\}$, which are given by the partition $\pi_m$, such that for all $i = 1, \cdots, m$,
        \begin{equation}\label{eq:joint_law_inductive}
            \Lc^\P(X_0, W^m_{t^m_i \wedge \cdot}, B^m_{t^m_i \wedge \cdot}, \mut^m_i, (\taut^m_j)_{j = 1}^{i})
            ~ = ~
            \Lc^\P(X_0, W^m_{t^m_i \wedge \cdot}, B^m_{t^m_i \wedge \cdot}, \mu^m_{t^m_{i}}, (\tau^m_j)_{j = 1}^{i}).
        \end{equation}
        By similar arguments of Lemma 4.7 in El Karoui and Tan \cite{ElKarouiTan2}, for $i = 1, \cdots, m$,there exists a sequence of Borel measurable functions
        $$
            G^\mu_i: \Cc^\ell \x \Pc_p(\Cc^n \x \Cc^d \x [0,T]) \x [0,1] \longrightarrow \Pc_p(\Cc^n \x \Cc^d \x [0,T]),
        $$
        such that, for any sequence of i.i.d uniform random variables $U^m = (U^m_i)_{1 \le i \le m}$ on $[0,1]$ independent of $B^m$, $W$, $X_0$ and $\mu^m$, one has that
        \begin{equation}\label{eq:G_i}
            \Lc^\P(B^m_{t^m_i \wedge \cdot}, \mu^m_{t^m_{i - 1}}, \widehat{G}^\mu_i)
            ~ = ~
            \Lc^\P(B^m_{t^m_i \wedge \cdot}, \mu^m_{t^m_{i - 1}}, \mu^m_{t^m_i}),
        \end{equation}
        where $\widehat{G}^\mu_i = G^\mu_i(B^m_{t^m_i \wedge \cdot}, \mu^m_{t^m_{i - 1}}, U^m_i)$.

        Similarly, for $i = 1, \cdots, m$, there exists a sequence of Borel measurable functions
        $$
            G^\tau_i: \R^n \x \Cc^d \x \Cc^\ell \x \Pc(\Cc^n \x [0,T]) \x [0,T]^{i - 1} \x [0,1] \longrightarrow [0,T],
        $$
        such that, for any sequence of i.i.d uniform random variables $V^m = (V^m_i)_{i = 1}^m$ on $[0,1]$ independent of $B^m$, $W^m$, $\mu^m$, $X_0$ and $\tau^m$, one has that
        \begin{equation}\label{eq:H_i}
            \Lc^\P(X_0, W^m_{t^m_i \wedge \cdot}, B^m_{t^m_i \wedge \cdot}, \mu^m_{t^m_{i}}, (\tau^m_j)_{j = 1}^{i - 1}, \widehat{G}^\tau_i)
            ~ = ~
            \Lc^\P(X_0, W^m_{t^m_i \wedge \cdot}, B^m_{t^m_i \wedge \cdot}, \mu^m_{t^m_{i}}, (\tau^m_j)_{j = 1}^{i - 1}, \tau^m_i),
        \end{equation}
        where we take $\tau^m_j := t^m_j \mathds{1}_{\{t^m_j\}}(\tau^m)$, for $j = 1, \cdots, m$, define for convention that $(\tau^m_j)_{j = 1}^0 := 0$ ,
        and $\widehat{G}^\tau_i := G^\tau_i(X_0, W^m_{t^m_i \wedge \cdot}, B^m_{t^m_i \wedge \cdot}, \mu^m_{t^m_{i}}, (\tau^m_j)_{j = 1}^{i - 1}, V^m_i)$.

        Moreover, there exists some Borel measurable function $\gamma : \R^n \longrightarrow [0,1]$ such that
        $\gamma^W_i := \gamma(W_{it^m_1/m} - W_{(i - 1)t^m_1/m})$,
        $\gamma^B_i := \gamma(B_{it^m_1/m} - B_{(i - 1)t^m_1/m})$
        are  uniform distributed random variables for all $i = 1, 2, \cdots, m$.

        Based on the above preparations, we first define that $\mut^m_0 := \mu^m_0 = \mu_0$,
        $\taut^m_{0} := 0$,
        and then define inductively, for all $i = 1, \cdots, m$, that
        \begin{align*}
            \mut^m_i
             :=
            G^\mu_i(B^m_{t^m_i \wedge \cdot}, \mut^m_{i - 1}, \gamma^B_i),
            ~
            \taut^m_i
             :=
            G^\tau_i(X_0, W^m_{t^m_i \wedge \cdot}, B^m_{t^m_i \wedge \cdot}, \mut^m_{t^m_{i}}, (\taut^m_j)_{j = 1}^{i - 1}, \gamma^W_i).
        \end{align*}
        We may notice that
        $\mut^m_i$ is $\sigma(\gamma^B_1, \cdots, \gamma^B_i, B^m_{t^m_i \wedge \cdot})$-measurable
        and $\taut^m_i$ is $\sigma(X_0, \mut^m_i, W_{t^m_i \wedge \cdot}, B^m_{t^m_i\wedge\cdot})$-measurable
        for all $i = 1, \cdots, m$.

        \noindent\textbf{Step.2} 
        we can inductively verify \eqref{eq:joint_law_inductive}.

        In fact, in the case that $i = 1$, one has that
        \begin{align*}
            & ~ \Lc^\P(X_0, W^m_{t^m_1 \wedge \cdot}, B^m_{t^m_1 \wedge \cdot}, \mut^m_1, \taut^m_1)
            \\ ~ = & ~
            \Lc^\P(X_0, W^m_{t^m_1 \wedge \cdot}, B^m_{t^m_1 \wedge \cdot}, \mut^m_1,
            G^\tau_1(X_0, W^m_{t^m_1 \wedge \cdot}, B^m_{t^m_1 \wedge \cdot}, \mut^m_1, 0, \gamma^W_1))
            \\ ~ = & ~
            \Lc^\P(X_0, W^m_{t^m_1 \wedge \cdot}, B^m_{t^m_1 \wedge \cdot}, G^\mu_1(B^m_{t^m_1\wedge \cdot}, \mu^m_0, U^m_1),
                \\ & ~~~~~~~~~~~~~~~
                 G^\tau_1(X_0, W^m_{t^m_1 \wedge \cdot}, B^m_{t^m_1 \wedge \cdot}, G^\mu_1(B^m_{t^m_1\wedge \cdot}, \mu^m_0, U^m_1), 0, \gamma^W_1))
            \\ ~ = & ~
            \Lc^\P(X_0, W^m_{t^m_1 \wedge \cdot}, B^m_{t^m_1 \wedge \cdot}, \mu^m_{t^m_1},
            H_1(X_0, W^m_{t^m_1 \wedge \cdot}, B^m_{t^m_1 \wedge \cdot}, \mu^m_{t^m_1}, 0, \gamma^W_1))
            \\ ~ = & ~
            \Lc^\P(X_0, W^m_{t^m_1 \wedge \cdot}, B^m_{t^m_1 \wedge \cdot}, \mu^m_{t^m_{1}}, \tau^m_1),
        \end{align*}
        where the first two equalities follow by the definitions of $\taut^m_1$ and $\mut^m_1$
        and the last two equalities follow by \eqref{eq:G_i} and \eqref{eq:H_i}.

        Next, we assume that \eqref{eq:joint_law_inductive} holds true for some $i = 1, \cdots, m-1$.

        We may first claim that
        \begin{equation*}
            \Lc^\P(X_0, W^m_{t^m_{i + 1} \wedge \cdot}, B^m_{t^m_{i + 1} \wedge \cdot},
                \mu^m_{t^m_{i + 1}},
                (\tau^m_{j})_{j = 1}^i)
            ~ = ~
            \Lc^\P(X_0, W^m_{t^m_{i + 1} \wedge \cdot}, B^m_{t^m_{i + 1} \wedge \cdot},
                \mut^m_{i + 1},
                (\taut^m_{j})_{j = 1}^{i}).
        \end{equation*}
        The difference of this equality and \eqref{eq:joint_law_inductive} is that if we replace $i$ in $(\tau^m_{j})_{j = 1}^i$, and $(\taut^m_{j})_{j = 1}^{i}$ with $i + 1$, then this equality becomes \eqref{eq:joint_law_inductive}.

        In fact, for any
        $\phi_1 \in C_b(\R^n \x \Cc^d \x \Cc^\ell \x \Pc_p(\Cc^n \x \Cc^d \x [0,T]) \x [0,T]^i),
        \phi_2 \in C_b(\Pc_p(\Cc^n \x \Cc^d \x [0,T])),
        \psi_1 \in C_b(\Cc^d),
        \psi_2 \in C_b(\Cc^\ell),$
        one has that
        \begin{align*}
            & ~
                \E^\P\Big[
                \phi_1\big(X_0,W^m_{t^m_{i} \wedge \cdot}, B^m_{t^m_{i} \wedge \cdot}, \mu^m_{t^m_i}, (\tau^m_j)_{j = 1}^i \big)
                \phi_2\big(\mu^m_{t^m_{i + 1}}\big)
                \psi_1\big(W^m_{t^m_{i + 1} \wedge \cdot} - W^m_{t^m_{i} \wedge \cdot}\big)
                \psi_2\big(B^m_{t^m_{i + 1} \wedge \cdot} - B^m_{t^m_{i} \wedge \cdot}\big)\Big]
            \\ = & ~
                \E^\P\Big[\E^\P\Big[
                \phi_1\big(X_0,W^m_{t^m_{i} \wedge \cdot}, B^m_{t^m_{i} \wedge \cdot}, \mu^m_{t^m_i}, (\tau^m_j)_{j = 1}^i \big)
                \psi_1\Big|B^m, \mu^m\Big]
                \phi_2\big(\mu^m_{t^m_{i + 1}}\big)
                \psi_2\Big]
            \\  = & ~
                \E^\P\Big[\E^\P\Big[
                \phi_1\big(X_0,W^m_{t^m_{i} \wedge \cdot}, B^m_{t^m_{i} \wedge \cdot}, \mu^m_{t^m_i}, (\tau^m_j)_{j = 1}^i \big)
                \psi_1\Big|B^m_{t^m_{i} \wedge \cdot}, \mu^m_{t^m_{i}}\Big]
                \phi_2\big(\mu^m_{t^m_{i + 1}}\big)
                \psi_2\Big]
            \\ = & ~
                \E^\P\Big[\E^\P\Big[
                \phi_1\big(X_0,W^m_{t^m_{i} \wedge \cdot}, B^m_{t^m_{i} \wedge \cdot}, \mu^m_{t^m_i}, (\tau^m_j)_{j = 1}^i \big)
                \psi_1\Big|B^m_{t^m_{i} \wedge \cdot}, \mu^m_{t^m_{i}}\Big]
                \phi_2\big(\Gh^\mu_{i + 1}\big)
                \psi_2\Big]
            \\ = & ~
                \E^\P\Big[
                \phi_1\big(X_0,W^m_{t^m_{i} \wedge \cdot}, B^m_{t^m_{i} \wedge \cdot}, \mut^m_{t^m_i}, (\taut^m_j)_{j = 1}^i \big)
                \phi_2\big(\mut^m_{t^m_{i + 1}}\big)
                \psi_1
                \psi_2\Big],
        \end{align*}
        where the first equality follows by tower property of expectation, the second one by \eqref{eq:H_hypothesis_m}, the third one by
        \eqref{eq:G_i} and the last one by the following argument:

        For any $\varphi_1 \in C_b(\Cc^\ell \x \Pc_p(\Cc^n \x \Cc^d \x [0,T])), \varphi_2 \in C_b([0,1])$
        \begin{align*}
           & ~
                \E^\P\Big[\E^\P\Big[
                \phi_1\big(X_0,W^m_{t^m_{i} \wedge \cdot}, B^m_{t^m_{i} \wedge \cdot}, \mu^m_{t^m_i}, (\tau^m_j)_{j = 1}^i \big)
                \psi_1\Big|B^m_{t^m_{i} \wedge \cdot}, \mu^m_{t^m_{i}}\Big]
                \varphi_1(B^m_{t^m_i \wedge \cdot}, \mu^m_{t^m_i})
                \varphi_2(U^m_{i + 1})
                \psi_2\Big]
           \\ = & ~
                \E^\P\Big[\E^\P\Big[
                \phi_1\big(X_0,W^m_{t^m_{i} \wedge \cdot}, B^m_{t^m_{i} \wedge \cdot}, \mu^m_{t^m_i}, (\tau^m_j)_{j = 1}^i \big)
                \psi_1
                \varphi_1(B^m_{t^m_i \wedge \cdot}, \mu^m_{t^m_i})
                \Big|B^m_{t^m_{i} \wedge \cdot}, \mu^m_{t^m_{i}}\Big]\psi_2\Big]
                \E^\P\Big[\varphi_2(U^m_{i + 1})\Big]
           \\ = & ~
                \E^\P\Big[\E^\P\Big[
                \phi_1\big(X_0,W^m_{t^m_{i} \wedge \cdot}, B^m_{t^m_{i} \wedge \cdot}, \mu^m_{t^m_i}, (\tau^m_j)_{j = 1}^i \big)
                \psi_1
                \varphi_1(B^m_{t^m_i \wedge \cdot}, \mu^m_{t^m_i})
                \Big|B^m, \mu^m\Big]\psi_2\Big]
                \E^\P\Big[\varphi_2(U^m_{i + 1})\Big]
           \\ = & ~
                \E^\P\Big[
                \phi_1\big(X_0,W^m_{t^m_{i} \wedge \cdot}, B^m_{t^m_{i} \wedge \cdot}, \mu^m_{t^m_i}, (\tau^m_j)_{j = 1}^i \big)
                \psi_1
                \varphi_1(B^m_{t^m_i \wedge \cdot}, \mu^m_{t^m_i})
                \psi_2\Big]
                \E^\P\Big[\varphi_2(U^m_{i + 1})\Big]
           \\ = & ~
                \E^\P\Big[
                \phi_1\big(X_0,W^m_{t^m_{i} \wedge \cdot}, B^m_{t^m_{i} \wedge \cdot}, \mut^m_{_i}, (\taut^m_j)_{j = 1}^i \big)
                \psi_1
                \varphi_1(B^m_{t^m_i \wedge \cdot}, \mut^m_{t^m_i})
                \psi_2\Big]
                \E^\P\Big[\varphi_2(\gamma^B_{i + 1})\Big]
           \\ = & ~
                \E^\P\Big[
                \phi_1\big(X_0,W^m_{t^m_{i} \wedge \cdot}, B^m_{t^m_{i} \wedge \cdot}, \mut^m_{_i}, (\taut^m_j)_{j = 1}^i \big)
                \psi_1
                \varphi_1(B^m_{t^m_i \wedge \cdot}, \mut^m_{t^m_i})
                \varphi_2(\gamma^B_{i + 1})
                \psi_2\Big],
        \end{align*}
        where the first equality follows by $U^m_{i + 1}$ independent of $(B^m_{t^m_{i + 1}\wedge \cdot}, \mu^)$,
        the second one by \eqref{eq:H_hypothesis_m},
        the fourth one by the inductive assumption
        and the last one by $\gamma^B_{i + 1}$ independent of $(X_0, W, B^m)$.

        Now, we can turn to prove \eqref{eq:joint_law_inductive} holds true for $i = m$.
        The arbitrariness of $\varphi_1$ and $\varphi_2$ allow us to replace
        $\varphi_1(B^m_{t^m_i \wedge \cdot}, \mu^m_{t^m_i})
                \varphi_2(U^m_{i + 1})$
        and
        $\varphi_1(B^m_{t^m_i \wedge \cdot}, \mut^m_{t^m_i})
                \varphi_2(\gamma^B_{i + 1})$
        with $\phi_2(\Gh^\mu_{i + 1})$ and $\phi_2\big(\mut^m_{t^m_{i + 1}}\big)$, respectively.
        Then it follows that
        \begin{align*}
            & ~ \Lc^\P(X_0, W^m_{t^m_{i + 1} \wedge \cdot}, B^m_{t^m_{i + 1} \wedge \cdot},
                \mut^m_{i + 1},
                \taut^m_{i + 1})
            \\ ~ = & ~
            \Lc^\P
            (X_0, W^m_{t^m_{i + 1} \wedge \cdot}, B^m_{t^m_{i + 1} \wedge \cdot},
                \mut^m_{i + 1},
                G^\tau_{i + 1}(X_0, W^m_{t^m_{i + 1} \wedge \cdot},
                B^m_{t^m_{i + 1} \wedge \cdot},
                    \mut^m_{i + 1}, (\taut^m_j)_{j = 1}^{i}, \gamma^W_{i + 1}))
            \\ ~ = & ~
            \Lc^\P
            (X_0, W^m_{t^m_{i + 1} \wedge \cdot}, B^m_{t^m_{i + 1} \wedge \cdot},
                \mu^m_{i + 1},
                G^\tau_{i + 1}(X_0, W^m_{t^m_{i + 1} \wedge \cdot},
                B^m_{t^m_{i + 1} \wedge \cdot},
                    \mu^m_{i + 1}, (\tau^m_j)_{j = 1}^{i}, V^m_{i+1}))
            \\ ~ = & ~
            \Lc^\P(X_0, W^m_{t^m_{i + 1} \wedge \cdot}, B^m_{t^m_{i + 1} \wedge \cdot},
                \mu^m_{t^m_{i + 1}},
                \tau^m_{i + 1}).
        \end{align*}
        Now we complete the proof of \eqref{eq:joint_law_inductive}.
        Further, we define $\taut^m := \sum_{i = 1}^{m}\taut^m_i$, and rewrite \eqref{eq:joint_law_inductive} as
        \begin{equation}\label{eq:joint_law}
            \Lc^\P(X_0, W^m_{t^m_i \wedge \cdot}, B^m_{t^m_i \wedge \cdot}, \mut^m_i, \taut^m \wedge t^m_i)
            ~ = ~
            \Lc^\P(X_0, W^m_{t^m_i \wedge \cdot}, B^m_{t^m_i \wedge \cdot}, \mu^m_{t^m_{i}}, \tau^m \wedge t^m_i).
        \end{equation}
        It is easy to verify that $\taut^m$ is a $\F^{X_0,W, B}$-stopping time.

        \noindent \textbf{Step.3}
        Let us define $\mut^m_t := \mut^m_m \circ (X_{t \wedge \cdot},W,\tau \wedge t)^{-1}$,
        and claim that 
        \begin{equation}\label{eq:law_equivalence_m}
            \Lc^\P(\Xt^m, W^m, B^m, \mut^m_T, \taut^m)
            ~ = ~
            \Lc^\P(X^m, W^m, B^m, \mu^m_T, \tau^m).
        \end{equation}
        Since $\Xt^m$ and $X^m$ are the solutions of SDE \eqref{eq:SDE_strong_discrete} and \eqref{eq:SDE_weak_discrete}, respectively,
        Assumption \ref{ass:coefficient_reward} implies that
        there exists a Borel measurable function
        $H: [0,T] \x \R^n \x \Cc^d \x \Cc^\ell \x \Pc_p(\Cc^n \x \Cc^d \x [0,T]) \x [0,T] \longrightarrow \Cc^n$ such that
        \begin{equation*}
            X^m_{t \wedge \cdot}
            ~ = ~
            H(t, X_0, W^m_{t\wedge \cdot}, B^m_{t\wedge \cdot}, \mu^m_t, \tau^m \wedge t),
        \end{equation*}
        and by \eqref{eq:joint_law_inductive}, one has that
        \begin{equation*}
            \Xt^m_{t \wedge \cdot}
            ~ = ~
            H(t, X_0, W^m_{t\wedge \cdot}, B^m_{t\wedge \cdot}, \mut^m_t, \taut^m \wedge t).
        \end{equation*}
        Then we can conclude that
        \begin{equation*}
            \Lc^\P(\Xt^m, W^m, B^m, \mut^m_T, \taut^m)
            ~ = ~
            \Lc^\P(X^m, W^m, B^m, \mu^m_T, \tau^m).
        \end{equation*}
        \textbf{Step.4}
        It remains to verify that for all $t \in [0,T]$,
        \begin{equation*}
            \mut^m_t = \Lc^\P(\Xt^m_{t\wedge\cdot}, W^m, \taut^m \wedge t|B).
        \end{equation*}
        In fact, for all $\phi \in C_b(\Cc^n \x [0,T])$ and $\varphi \in C_b(\Cc^\ell \x \Pc(\Cc^n \x [0,T]))$,
        it follows that
        \begin{align*}
            & ~
                \E^\P[\langle\phi,\mut^m_t\rangle\varphi(B^m, \mut^m_T)]
            ~ = ~
                \E^\P[\langle\phi,\mu^m_t\rangle\varphi(B^m, \mu^m_T)]
            \\ = & ~
                \E^\P[\phi(X^m_{t\wedge \cdot}, W^m, \tau^m \wedge t)\varphi(B^m, \mu^m_T)]
            ~ = ~
                \E^\P[\phi(\Xt^m_{t\wedge \cdot}, W^m, \taut^m \wedge t)\varphi(B^m, \mut^m_T)]
            \\ = & ~
                \E^\P[\E^\P[
                \phi(\Xt^m_{t\wedge \cdot}, W^m, \taut^m \wedge t)|B^m, \mut^m_T]
                \varphi(B^m, \mut^m_T)],
        \end{align*}
        where the first and third inequalities follow by \eqref{eq:law_equivalence_m}, the second one by \eqref{eq:consistency_m},
        i.e.
        \begin{equation*}
            \mut^m_t
            ~ = ~
            \Lc^\P(\Xt^m_{t\wedge\cdot}, W^m, \taut^m \wedge t|B^m, \mut^m_T).
        \end{equation*}
        Recall that 
        $\mut^m_i$ is $\sigma(\gamma^B_1, \cdots, \gamma^B_i, B^m_{t^m_i \wedge \cdot})$-measurable,
        $\taut^m_i$ is $\sigma(X_0, \mut^m_i, W_{t^m_i \wedge \cdot}, B^m_{t^m_i\wedge\cdot})$-measurable,
        for all $i = 1, \cdots, m$.
        More importantly, for each $i = 1, \cdots, m$,
        there exists some Borel measurable function 
        $ 
        G^B_i : \Cc^\ell \x \Pc_p(\Cc^n \x \Cc^d \x [0,T]) \x \Pc_p(\Cc^n \x \Cc^d \x [0,T]) \longrightarrow \Cc^\ell
        $,
        such that
        $$
        B_{it^m_1/m} - B_{(i - 1)t^m_1/m}
        :=
        G^B_i(B^m_{t^m_i \wedge \cdot}, \mut^m_{i - 1}, \mut^m_{i}).
        $$
        That means conditioning on $B^m, \mut^m_T$ is equivalent to conditioning on $B$.
        We can then conclude that \eqref{eq:consistency_m_strong} holds true.
    \end{proof}

\subsection{Propagation of chaos for stopped McKean-Vlasov SDE}

    In this subsection, we aim to prove that under proper conditions, it follows that
    \begin{equation*}
        V_S(\nu) \le \varliminf_{N \to \infty}V^N_S(\nu_1, \cdots, \nu_N).
    \end{equation*}
    Our main tool is a strong propagation of chaos result for stopped McKean-Vlasove SDE in the sense that for any strong solution of the stopped McKean-Vlasov SDE, there exists a sequence of empirical law of strong solutions of the large population stopped SDEs given by \eqref{eq:N_SDE} that converges to its joint law in $\Wc_p$.

    \begin{proposition}\label{prop:propagation_of_chaos}
        Let Assumption \ref{ass:coefficient_reward} holds true.
        For any $\P \in \Pc_S(\nu)$ with a sequence of $(\nu_i)_{i \in \N} \subset \Pc_{p}(\R^n)$ such that for some $\eps > 0$,
        \begin{equation*}
            \sup_{N}
            \frac{1}{N}\sum_{i = 1}^{N}
            \int_{\R^n}
                |x|^{p + \eps}
            v_i(dx)
            < +\infty,
            \quad
            \lim_{N \to \infty}
            \Wc_p\Big(\frac{1}{N}\sum_{i = 1}^{N}\nu_i, \nu\Big)
            ~ = ~
            0,
        \end{equation*}
        if there exists some continuous function $\varphi: \R^n \x \Cc^d \x \Cc^\ell \longrightarrow [0,T]$, such that $\tau = \varphi(X_0, W, B)$ $\P$-a.s.,
        then
        \begin{equation*}
            \lim_{N \to \infty}
            \E^{\P^0}\bigg[
                \Wc_p\Big(\frac{1}{N}\sum_{i = 1}^{N}\delta_{(X^{i, N, \taub^{N}}, W^i, \tau^i)}, \Lc^\P((X, W, \tau)|B)\Big)
            \bigg]
            ~ = ~
            0,
        \end{equation*}
        where $\tau^i := \varphi(X^i_0, W^i, B)$, $\taub^{N} := (\tau^1, \cdots, \tau^N)$ and $X^{N, \taub^{N}} = (X^{1, N, \taub^{N}}, \cdots, X^{N, N, \taub^{N}})$ is the unique strong solution to the large population stopped SDE given by \eqref{eq:N_SDE}.

        Consequently, it holds that
        \begin{equation*}
            \lim_{N \to \infty}\Wc_p(\P_N, \P) ~ = ~ 0,
        \end{equation*}
        where $\{\P_N\}_{N \in \N_+}$ are given as in \eqref{eq:P_N_construct} with $\taub^{N} = (\tau^1, \cdots, \tau^N)$ above.
    \end{proposition}

    \begin{proof}
        For $\tau^i = \varphi(X^i_0, W^i, B)$, we define a sequence of stopped SDEs $(\Xo^{i,\tau^{i}})_{i \in \N}$ as follows:
        for each $i \in \N$,
        \begin{equation}
            \begin{split}
            \Xo^{i,\tau^i}_t
            ~ = ~
            X^i_0 ~ + ~ \int_{0}^{t \wedge \tau^i}b(s, \Xo^{i,\tau^i}_{s\wedge \cdot}, \mu_s) dt
            ~ & + ~ \int_{0}^{t \wedge \tau^i}\sigma(s, \Xo^{i,\tau^i}_{s\wedge \cdot}, \mu_s) dW^i_s
            \\~ & + ~ \int_{0}^{t \wedge \tau^i}\sigma_0(s, \Xo^{i,\tau^i}_{s\wedge \cdot}, \mu_s) dB_s,
            ~ t \in [0,T],
            \end{split}
        \end{equation}
        where $\mu_s := \Lc^\P((X_{s \wedge \cdot}, W, \tau \wedge s)|B_{s \wedge \cdot})$.

        For simplicity,
        we may write $\muo^N := \frac{1}{N}\sum_{i = 1}^{N}\delta_{(\Xo^{i,\tau^i}, W^i, \tau^i)}$,
        then we claim that
        \begin{align}\label{eq:approximation_copy}
            & ~
            \lim_{N \to \infty}
            \E^{\P^0}\bigg[
                \Wc^p_p\Big(\muo^N, \Lc^\P(X, W, \tau|B)\Big)
            \bigg]
            ~ = ~
            0,
            \\ \label{eq:approximation_N_copy}
            & ~
            \lim_{N \to \infty}
            \E^{\P^0}\bigg[
                \Wc^p_p\Big(\mu^{N, \taub^N}_T, \Lc^\P(X, W, \tau|B) \Big)
            \bigg]
            ~ = ~
            0,
        \end{align}
        where 
        \begin{align*}
            \mu^{N, \taub^N}_T
            ~ := ~
            \frac{1}{N}\sum_{i = 1}^{N}\delta_{(X^{i, N, \taub^{N}}, W^i, \tau^i)}
        \end{align*}

        For the proof of \eqref{eq:approximation_copy},
        it is sufficient to prove that in the space $\Pc_p(\Pc^2_p(\Cc^n \x \Cc^d \x [0,T]))$ equipped with Wasserstein-$p$ distance,
        where $\Pc^2_p(\Cc^n \x \Cc^d \x [0,T])$ is equipped with $1$-product metric.
        \begin{equation}\label{eq:local_approximation_copy}
            \lim_{N \to \infty}
            \Wc_p
            \Big(\Lc^{\P^0}(\muo^N,\mu_T),\Lc^{\P^0}(\mu_T,\mu_T)\Big) ~ = ~ 0.
        \end{equation}
        In fact, we can observe that
        \begin{align*}
            \E^{\P^0}\bigg[
                \Wc^p_p\Big(\muo^N, \Lc^\P(X, W, \tau|B)\Big)
            \bigg]
            ~ = ~
            \int_{\Pc^2_p(\Cc^n \x \Cc^d \x [0,T])}
                \Wc^p_p(\mu,\nu)
            \Lc^{\P^0}(\muo^N,\mu_T)
            (d\mu,d\nu),
        \end{align*}
        and that 
        \begin{align*}
            \Wc^p_p(\mu,\nu) 
            \le 
            \big(\Wc_p(\mu,\delta_{(0,0,0)}) + \Wc_p(\nu,\delta_{(0,0,0)})\big)^p
            ~ \le ~
            1 + \big(\Wc_p(\mu,\delta_{(0,0,0)}) + \Wc_p(\nu,\delta_{(0,0,0)})\big)^p.
        \end{align*}
        Here we let $\varphi = \Wc^p_p$ to apply $(iv)$ in Theorem 7.12 of \cite{Villani2021}.
        
        The proof of \eqref{eq:local_approximation_copy} is divided into two steps, 
        where the first one is the relatively compactness of $\{\Lc^{\P^0}(\muo^N,\mu_T)\}_{N \in \N_+}$,
        and the second one is that any limit point of the sequence $\{\Lc^{\P^0}(\muo^N,\mu_T)\}_{N \in \N_+}$ is identical to $\Lc^{\P^0}(\mu_T,\mu_T)$.

        For the relatively compactness of $\{\Lc^{\P^0}(\muo^N,\mu_T)\}_{N \in \N_+}$, 
        we first prove its tightness,which is equivalent to the tightness of all its marginal measures, and consider their mean measures $\{m\Lc^{\P^0}(\muo^N)\}_{N \in \N_+} \subset \Pc_p(\Cc^n \x \Cc^d \x [0,T])$,
        where for each $\P \in \Pc_p(\Pc^2_p(\Cc^n \x \Cc^d \x [0,T]))$,
        its mean measure $m\P \in \Pc_p(\Cc^n \x \Cc^d \x [0,T])$ is defined by
        \begin{equation*}
            m\P(A)
            ~ := ~
            \int_{\Pc_p(\Cc^n \x \Cc^d \x [0,T])}
                \mu(A)
            \P(d\mu),
            ~\mbox{for all}~
            A \in \Bc(\Cc^n \x \Cc^d \x [0,T]).
        \end{equation*}
        Then again, we consider all the marginal measures of $\{m\Lc^{\P^0}(\muo^N)\}_{N \in \N_+}$, and get that
        \begin{align*}
            m\Lc^{\P^0}(\muo^N)
            ~ = ~ \frac{1}{N}\sum_{i = 1}^{N}
                \Lc^{\P^0}(\Xo^{i,\tau^i}, W^i, \tau^i).
        \end{align*}
        Thus, the tightness of $\{m\Lc^{\P^0}(\muo^N)(\Cc^n, \cdot, [0,T])\}_{N \in \N_+}$ and $\{m\Lc^{\P^0}(\muo^N)(\Cc^n, \Cc^d, \cdot)\}_{N \in \N_+}$ hold true 
        since we have that 
        \begin{align*}
            m\Lc^{\P^0}(\muo^N)(\Cc^n, \cdot, [0,T])
            ~ = ~
            \frac{1}{N}\sum_{i = 1}^N\Lc^{\P^0}(W^i),
            \quad
            m\Lc^{\P^0}(\muo^N)(\Cc^n, \Cc^d, \cdot)
            ~ = ~
            \frac{1}{N}\sum_{i = 1}^N\Lc^{\P^0}(\tau^i).
        \end{align*}

        For the tightness of $\{m\Lc^{\P^0}(\muo^N)(\cdot, \Cc^d, [0,T])\}_{N \in \N_+}$ in $\Pc_p(\Cc^n)$,
        one has the standard estimations, for some constant $C > 0$,
        \begin{align*}
            & ~
                \sup_{N \in \N_+}
                m\Lc^{\P^0}(\muo^N)(\cdot, \Cc^d, [0,T])
                \bigg[\sup_{t\in [0,T]}|X_t| \ge a\bigg]
                =
                \sup_{N \in \N_+}
                \frac{1}{N}\sum_{i = 1}^{N}\P^0
                \bigg[
                \sup_{t\in [0,T]}|\Xo^i_t| \ge a\bigg]
            \\ \le &   ~
                \sup_{N \in \N_+}
                    \frac{1}{Na^p}\sum_{i = 1}^{N}\E^{\P^0}\bigg[\sup_{t\in [0,T]}|\Xo^i_t|^p\bigg]
                \le
                \sup_{N \in \N_+}
                    \frac{1}{Na^p}\sum_{i = 1}^{N}C\E^{\P^0}[1 + |X^{i}_0|^p]
            \\ = &   ~    
                \frac{C}{a^p}\bigg(1 + \sup_{N \in \N_+}
                    \frac{1}{N}\sum_{i = 1}^{N}\int_{\R^n}|x|^p\nu_i(dx)\bigg)
        \end{align*}
        and
        \begin{align*}
            & ~
                \sup_{\tau \in \Tc^N}
                m\Lc^{\P^0}(\muo^N)(\cdot, \Cc^d, [0,T])[|X_{(\tau + \delta)\wedge T} - X_\tau| \ge a]
            \\ \le &  ~
                \sup_{\tau \in \Tc^N}
                    \frac{1}{Na^p}\sum_{i = 1}^{N}E^{\P^0}[|\Xo^i_{(\tau + \delta)\wedge T} - \Xo^i_\tau|^p]
            \\ \le &  ~
            \frac{C}{a^p}\bigg(1 + \sup_{N \in \N_+}
            \frac{1}{N}\sum_{i = 1}^{N}\int_{\R^n}|x|^p\nu_i(dx)\bigg)\delta.
        \end{align*}
        Then by Aldous' criterion (see Theorem 16.10 in Billingsley \cite{Billingsley}),  we prove the tightness of $\{m\Lc^{\P^0}(\muo^N)(\cdot, \Cc^d, [0,T])\}_{N \in \N_+}$ in $\Pc_p(\Cc^n)$.

        Therefore, we have the tightness of $\{\Lc^{\P^0}(\muo^N)\}_{N \in \N_+}$, 
        thus tightness of $\{\Lc^{\P^0}(\muo^N,\mu_T)\}_{N \in \N_+}$, whose relatively compactness in $\Pc_p(\Pc^2_p(\Cc^n \x \Cc^d \x [0,T]))$ will immediately hold true if we have the uniformly integrability of $\{\Lc^{\P^0}(\muo^N,\mu_T)\}_{N \in \N_+}$ in the sense that
        \begin{align*}
            & \lim_{a \to \infty}
                \sup_{N \in \N_+}
                \E^{\Lc^{\P^0}(\muo^N,\mu_T)}\bigg[\big(\Wc_p(\mu_1,\delta_0) + \Wc_p(\mu_2,\delta_0)\big)^p
                \mathds{1}_{\{(\Wc_p(\mu_1,\delta_0) + \Wc_p(\mu_2,\delta_0)) \ge a\}}\bigg]
            \\ = & 
            \lim_{a \to \infty}
                \sup_{N \in \N_+}
                \E^{\P^0}
                \bigg[\big(\Wc_p(\muo^N,\delta_0) + \Wc_p(\mu_T,\delta_0)\big)^p
                \mathds{1}_{\{(\Wc_p(\muo^N,\delta_0) + \Wc_p(\mu_T,\delta_0)) \ge a\}}\bigg]
            ~ = ~ 0.
        \end{align*}
        In fact, one has the estimation for some constant $C>0$,
        \begin{align*}
            & ~
            \sup_{N \in \N_+}
                \E^{\P^0}
                \bigg[\big(\Wc_p(\muo^N,\delta_0) + \Wc_p(\mu_T,\delta_0)\big)^p
                \mathds{1}_{\{(\Wc_p(\muo^N,\delta_0) + \Wc_p(\mu_T,\delta_0)) \ge a\}}\bigg]
            \\ \le & ~
            \sup_{N \in \N_+}
                \frac{C}{a^\eps}\E^{\P^0}
                [\Wc^{p + \eps}_p(\muo^N,\delta_0)] 
                + 
                C\E^{\P^0}[\Wc^{p}_p(\mu_T,\delta_0)\mathds{1}_{\{\Wc_p(\mu_T,\delta_0) \ge a\}}]
            \\ \le & ~
            \sup_{N \in \N_+}
                \frac{C}{a^\eps}\E^{\P^0}
                \Big[\|\Xo^i\|^{p + \eps} + \|W^i\|^{p + \eps} + |\tau^i|^{p + \eps}\Big]
                +
                C\E^{\P^0}[\Wc^{p}_p(\mu_T,\delta_0)\mathds{1}_{\{\Wc_p(\mu_T,\delta_0) \ge a\}}]
            \\ \le & ~
                \frac{C}{a^\eps}\bigg(1 + \sup_{N \in \N_+}
                \frac{1}{N}\sum_{i = 1}^{N}\int_{\R^n}|x|^{p + \eps}\nu_i(dx)\bigg)
                + C\E^{\P^0}[\Wc^{p}_p(\mu_T,\delta_0)\mathds{1}_{\{\Wc_p(\mu_T,\delta_0) \ge a\}}].
        \end{align*}
        Thus, we can assume WLOG that there exists some $\Lc^{\P^0}(\muo^\infty,\mu_T) \in \Pc_p(\Pc^2_p(\Cc^n \x \Cc^d \x [0,T]))$ such that
        \begin{equation*}
            \lim_{N \to \infty}
            \Wc_p(\Lc^{\P^0}(\muo^N,\mu_T),\Lc^{\P^0}(\muo^\infty,\mu_T))
            ~ = ~
            0,
        \end{equation*}
        and it remains to verify that $\Lc^{\P^0}(\muo^\infty,\mu_T) = \Lc^{\P^0}(\mu_T,\mu_T)$,
        which, by Proposition A.3 in \cite{DjetePossamaiTan}, is equivalent to that for any $k \in \N_+$, and $f, h_1, \cdots, h_k \in C_b(\Cc^n \x \Cc^d \x [0,T])$, one has that
        \begin{align*}
            &\int_{\Pc^2(\Cc^n \x \Cc^d \x [0,T])}
                \prod_{i = 1}^k \langle h_i, \mu_1 \rangle
                \langle f, \mu_2 \rangle
                \Lc^{\P^0}(\muo^\infty,\mu_T)(d\mu_1,d\mu_2)
            \\~ = &~
            \int_{\Pc^2(\Cc^n \x \Cc^d \x [0,T])}
            \prod_{i = 1}^k 
                \langle h_i, \mu \rangle
                \langle f, \mu_2 \rangle
                \Lc^{\P^0}(\mu_T,\mu_T)(d\mu_1,d\mu_2).
        \end{align*}
        In fact, it is sufficient to prove the claim for the case $k = 2$, since the proof for arbitrary $k$ is identical, then it follows that
        \begin{align*}
            & ~
            \int_{\Pc^2(\Cc^n \x \Cc^d \x [0,T])}
                \prod_{i = 1}^2 
                \langle h_i, \mu_1 \rangle
                \langle f, \mu_2 \rangle
                \Lc^{\P^0}(\muo^\infty,\mu_T)(d\mu_1,d\mu_2)
            \\ = & ~
            \lim_{N \to \infty}
            \int_{\Pc^2(\Cc^n \x \Cc^d \x [0,T])}
                \prod_{i = 1}^2 
                \langle h_i, \mu_1 \rangle
                \langle f, \mu_2 \rangle
                \Lc^{\P^0}(\muo^N,\mu_T)(d\mu_1,d\mu_2)
            \\ = & ~
            \lim_{N \to \infty}
                \frac{1}{N^2}\sum_{i,j = 1}^{N}
                \E^{\P^0}\Big[
                    h_1(\Xo^{i,\tau^i}, W^i, \tau^i)
                    h_2(\Xo^{j,\tau^j}, W^j, \tau^j)
                    \E^\P[f(X,W,\tau)|B]\Big]
            \\ = & ~
            \lim_{N \to \infty}
            \frac{1}{N^2}\sum_{i,j = 1}^{N}
            \E^{\P^0}\Big[
                \E^{\P^0}\Big[h_1(\Xo^{i,\tau^i}, W^i, \tau^i)|B\Big]
                \E^{\P^0}\Big[h_2(\Xo^{j,\tau^j}, W^j, \tau^j)|B\Big]
                \E^\P[f(X,W,\tau)|B]\Big]
            \\ = & ~
            \lim_{N \to \infty}
            \int_{\Pc^2(\Cc^n \x \Cc^d \x [0,T])}
                \prod_{j = 1}^2 
                \langle h_j, \mu_1 \rangle
                \langle f, \mu_2 \rangle
                \Lc^{\P^0}
                    \bigg(\frac{1}{N}\sum_{i = 1}^{N}\Lc^{\P^0}(\Xo^{i,\tau^i}, W^i, \tau^i|B),
                    \mu_T\bigg)(d\mu_1,\mu_2).
        \end{align*}
        Let $\Xo_{0,N}$ be a $\R^n$-valued random variable on $(\Om, \Fc_0, \P)$ with $\Lc^{\P}(\Xo_{0,N}) = \frac{1}{N}\sum_{i = 1}^{N}\nu_i$, independent of $(W, B)$. 
        Then if we denote by $\tau^N_N := \varphi(\Xo_{0,N}, W^i, B)$, one has that
        \begin{align*}
            \Lc^{\P^0}
                \bigg(\frac{1}{N}\sum_{i = 1}^{N}\Lc^{\P^0}(\Xo^{i,\tau^i}, W^i, \tau^i|B),
                \mu_T\bigg)
            ~ = ~
            \Lc^{\P}
                \big(\Lc^{\P}(\Xo^{\tau^N_N}, W, \tau^N_N|B),
                \mu_T\big),
        \end{align*}
        where $\Xo^{\tau^N_N}$ is the unique strong solution to the following SDE:
        \begin{equation*}
            \begin{split}
            \Xo^{\tau^N_N}_t
            ~ = ~
            \Xo_{0,N} ~ + ~ \int_{0}^{t \wedge \tau^N_N}b(s, \Xo^{\tau^N_N}_{s\wedge \cdot}, \mu_s) dt
            ~ & + ~ \int_{0}^{t \wedge \tau^N_N}\sigma(s, \Xo^{\tau^N_N}_{s\wedge \cdot}, \mu_s) dW_s
            \\~ & + ~ \int_{0}^{t \wedge \tau^N_N}\sigma_0(s, \Xo^{\tau^N_N}_{s\wedge \cdot}, \mu_s) dB_s,
            ~ t \in [0,T].
            \end{split}
        \end{equation*}
        Thus, it follows that
        \begin{align*}
            & ~
            \int_{\Pc^2(\Cc^n \x \Cc^d \x [0,T])}
                \prod_{j = 1}^2 
                \langle h_j, \mu_1 \rangle
                \langle f, \mu_2 \rangle
                \Lc^{\P^0}(\muo^\infty,\mu_T)(d\mu_1,d\mu_2)
            \\ = & ~
            \lim_{N \to \infty}
            \int_{\Pc^2(\Cc^n \x \Cc^d \x [0,T])}
                \prod_{j = 1}^2 
                \langle h_j, \mu_1 \rangle
                \langle f, \mu_2 \rangle
                \Lc^{\P}
                    \big(\Lc^{\P}(\Xo^{\tau^N_N}, W, \tau^N_N|B),
                    \mu_T\big)
                (d\mu_1,d\mu_2)
            \\ = & ~
            \lim_{N \to \infty}
                \E^{\P}\bigg[\prod_{j = 1}^2 \E^{\P}[h_j(\Xo^{\tau^N_N}, W, \tau^N_N)|B]
                \E^{\P}[f(X,W,\tau)|B]\bigg]
            \\ = & ~
            \E^{\P}\bigg[\prod_{j = 1}^2 \E^{\P}[h_j(X, W, \tau)|B]\E^{\P}[f(X,W,\tau)|B]\bigg]
            \\ = & ~
            \int_{\Pc^2(\Cc^n \x \Cc^d \x [0,T])}
                \prod_{j = 1}^2 
                \langle h_j, \mu_1 \rangle
                \langle f, \mu_2 \rangle
                \Lc^{\P^0}(\mu_T,\mu_T)(d\mu_1,d\mu_2),
        \end{align*}
        where the third equality follows from the dominated convergence theorem and that the sequence $\{\Lc^{\P}(\Xo^{\tau^N_N}, W, \tau^N_N|B)\}_{N \in \N_+}$ of probability measures converges weakly to $\Lc^{\P}(X,W,\tau|B)$, $\P$-a.s. as $N$ tends to $\infty$, which is a consequence of stable topology techniques if for all $\phi \in C_b(\Cc^n \x \Cc^d \x [0,T])$, we rewrite $\phi(\Xo^{\tau^N_N}, W, \tau^N_N)$ as a function of $\Xo_{0,N}$, $W$ and $B$ and observe that the function is continuous w.r.t. $\Xo_{0,N}$.
        
        For the proof of \eqref{eq:approximation_N_copy}, by a classic argument of SDE (very similar to the argument in Lemma \ref{lemma:equivalence_approximation} but easier), it follows that, for some constant $C > 0$ and each $i \in \N_+$,
        \begin{align*}
            & ~
            \E^{\P^0}\bigg[
                \sup_{t \in [0,T]}\big|\Xo^i_t - X^{i,N,\taub^{N}}_t\big|^p
                \bigg]
            ~ \le ~
            C\E^{\P^0}\bigg[\int_{0}^{T}\Wc^p_p(\mu^{N,\taub^{N}}_s, \mu_s)ds\bigg],
        \end{align*}
        where recall that for each $t \in [0,T]$,
        $$
            \mu^{N,\taub^{N}}_t
            ~ := ~
            \frac{1}{N}\sum_{i = 1}^{N}
            \delta_{\big(X^{i,N,\taub^{N}}_{t\wedge \cdot}, W^i, \tau^i\wedge t\big)},
        $$
        Again, recall that
        $
            \muo^N := \frac{1}{N}\sum_{i = 1}^{N}\delta_{(\Xo^{i,\tau^i}, W^i, \tau^i)},
        $
        so the difference of $\muo^N$ and $\mu^{N,\taub^{N}}_T$ lies only in the first marginal.
        Thus, we have the estimation that
        \begin{align*}
            & ~
                \E^{\P^0}\bigg[
                    \Wc^p_p\Big(\mu^{N,\taub^N}_T, \Lc^\P((X, W, \tau)|B) \Big)
                \bigg]
            \\ \le & ~
                C\E^{\P^0}\bigg[
                    \Wc^p_p\Big(\muo^N, \Lc^\P((X, W, \tau)|B)\Big)
                \bigg]
                +
                C\E^{\P^0}\bigg[
                    \Wc^p_p\Big(\mu^{N,\taub^N}_T, \muo^N\Big)
                \bigg]
            \\ \le & ~
                C\E^{\P^0}\bigg[
                    \Wc^p_p\Big(\muo^N, \Lc^\P((X, W, \tau)|B)\Big)
                \bigg]
                +
                C\E^{\P^0}\bigg[
                    \sup_{t \in [0,T]}\big|\Xo^i_t - X^{i,N,\taub^{N}}_t\big|^p
                    \bigg]
            \\ \le & ~
                C\E^{\P^0}\bigg[
                    \Wc^p_p\Big(\muo^N, \Lc^\P((X, W, \tau)|B)\Big)
                \bigg]
                +
                C\E^{\P^0}\bigg[\int_{0}^{T}\Wc^p_p(\mu^{N,\taub^{N}}_s, \mu_s)ds\bigg].
        \end{align*}
        Thus, by Grownwall's lemma, it follows that for some constant $C > 0$,
        \begin{equation*}
            \E^{\P^0}\bigg[
                    \Wc^p_p\Big(\mu^{N,\taub^N}_T, \Lc^\P((X, W, \tau)|B) \Big)
                \bigg]
            \le
            C\E^{\P^0}\bigg[
                    \Wc^p_p\Big(\muo^N, \Lc^\P((X, W, \tau)|B)\Big)
                \bigg].
        \end{equation*}
        Letting $N$ tends to $\infty$ on both sides, one can conclude the proof of \eqref{eq:approximation_N_copy} by \eqref{eq:approximation_copy}.
    \end{proof}

    \begin{proposition}\label{prop:limit_value_function}
        Let Assumption \ref{ass:coefficient_reward} holds true.
        For any $\nu \in \Pc_p(\R^n)$ with a sequence of $(\nu_i)_{i \in \N} \subset \Pc_p(\R^n)$ such thatfor some $\eps > 0$,
        \begin{equation*}
            \sup_{N}
            \frac{1}{N}\sum_{i = 1}^{N}
            \int_{\R^n}
                |x|^{p + \eps}
            v_i(dx)
            < +\infty,
            \quad
            \lim_{N \to \infty}
            \Wc_p\Big(\frac{1}{N}\sum_{i = 1}^{N}\nu_i, \nu\Big)
            ~ = ~
            0.
        \end{equation*}
        Then, it holds that
        \begin{equation*}
            V_S(\nu) \le \varliminf_{N \to \infty}V^N_S(\nu_1, \cdots, \nu_N),
            ~
            V_S(\nu) \le \varliminf_{N \to \infty}V_S\bigg(\frac{1}{N}\sum_{i = 1}^{N}\nu_i\bigg).
        \end{equation*}
    \end{proposition}

    \begin{proof}
        For any $\P \in \Pc_S(\nu)$, there exists a Borel measurable function $\varphi: \R^n \x \Cc^d \x \Cc^\ell \longrightarrow [0,T]$, such that $\tau = \varphi(X_0, W, B)$ $\P$-a.s.

        Then there exists a sequence of continuous functions $\{\varphi_m: \R^n \x \Cc^d \x \Cc^\ell \longrightarrow [0,T]\}_{m \in \N_+}$, such that
        for each $m \in \N_+$, $\tau_m = \varphi_m(X_0, W, B)$ is a $\F^{X_0, W, B}$-stopping time and
        \begin{equation*}
            \lim_{m \to \infty} \tau_m = \tau,
            ~
            \P-
            \mbox{a.s.}
        \end{equation*}
        In fact, we consider a $\F^{X_0, W, B}$-adapted, nondecreasing process $L := \mathds{1}_{\{\tau < \cdot\}}$, and a sequence of partitions $\{\pi_n\}_{n \in \N}$ of $[0,T]$ with $\pi_n : 0 = t^n_0 < t^n_1 < \cdots < t^n_{n} = T$, $|\pi_n| \le 1/n$, as well as their corresponding $\F^{X_0, W, B}$-adapted, nondecreasing processes $\{L^n\}_{n \in \N_+}$, where for each $n \in \N_+$,
        \begin{equation*}
            L^n_t 
            ~ := ~
            \sum_{i = 0}^{n - 1}
                L_{t^n_i}\mathds{1}_{(t^n_i,t^n_{i + 1}]}(t),
            ~\mbox{for all}~
            t \in [0,T].
        \end{equation*}
        Then there exists a family of continuous functions $\{\psi^n_i: \R^n \x \Cc^d \x \Cc^\ell \longrightarrow [0,1]\}_{n \in \N_+, i \in \{0, \cdots, n\}}$ such that 
        \begin{equation*}
            \E\Big[\Big|\psi^n_i(X_0,W_{t^n_i \wedge \cdot}, B_{t^n_i \wedge \cdot}) 
            -  
            L_{t^n_i}\Big|\Big] < \frac{1}{n^2}.
        \end{equation*}
        WLOG, we may assume that for each $n \in \N_+$,
        $
            \psi^n_0 \le \psi^n_1 \le \cdots, \le \psi^n_n,
        $
        otherwise we can replace $\psi^n_i$ by $\max_{j \in \{0, \cdots, i\}}\psi^n_j$.
        Then one may construct a sequence of $\F^{X_0, W, B}$-adapted, nondecreasing processes $\{\Lh^n\}_{n \in \N_+}$ by
        \begin{equation*}
            \Lh^n_t 
            ~ := ~
            \sum_{i = 0}^{n - 1}
                \psi^n_i(X_0,W_{t^n_i \wedge \cdot}, B_{t^n_i \wedge \cdot})\mathds{1}_{(t^n_i,t^n_{i + 1}]}(t),
            ~\mbox{for all}~
            t \in [0,T].
        \end{equation*}
        Thus, one obtains the estimation for each $\eps > 0$,
        \begin{align*}
            ~
            \P(d_L(L^n,\Lh^n) \ge \eps)
            ~
            \le &~
            \frac{1}{\eps}\E\Big[\max_{i \in \{0, \cdots, n\}}\Big|\psi^n_i(X_0,W_{t^n_i \wedge \cdot}, B_{t^n_i \wedge \cdot}) 
                -  
                L_{t^n_i}\Big|\Big]
            \\  ~
            \le &~
            \frac{1}{\eps}\sum_{i = 0}^{n}\E\Big[\Big|\psi^n_i(X_0,W_{t^n_i \wedge \cdot}, B_{t^n_i \wedge \cdot}) 
                -  
                L_{t^n_i}\Big|\Big]
              ~
            = ~ 
            \frac{n+ 1}{\eps n^2},
        \end{align*}
        where $d_L$ is the L\'evy metric on $[0,T]$, and concludes that
        \begin{equation*}
            \lim_{n \to \infty}\P(d_L(L,\Lh^n) \ge \eps)
            \le
            \lim_{n \to \infty}\P(d_L(L,L^n) \ge \eps/2)
            +
            \lim_{n \to \infty}\P(d_L(L^n,\Lh^n) \ge \eps/2)
            =
            0.
        \end{equation*}
        Therefore, if we define $\tau^v := \inf\{t\ge 0: L \ge \frac{1}{2}\}$ for nondecreasing functions $v$ on $[0,T]$ with $v_0 = 0$,
        and note that $v \longmapsto \tau^v$ is continuous at those $v$ that are strictly increasing,
        then there exists a subsequence $\{n_k\}_{k \in \N_+}$ of $\{n\}_{n \in \N_+}$,
        such that 
        \begin{equation*}
            \lim_{k \to \infty}\tau^{\Lh^{n_k}+\frac{\cdot}{3T}}
            ~ = ~
            \tau^{L+\frac{\cdot}{3T}}
            ~ = ~
            \tau,
            ~\P-\mbox{a.s.,}
        \end{equation*}
        Finally, we can conclude the claim for $\tau$, since for each $n \in \N_+$, 
        $\tau^{\Lh^n+\frac{\cdot}{3T}}$ is a $\F^{X_0,W,B}$-stopping time and there exists a continuous function $\varphi_n: \R^n \x \Cc^d \x \Cc^\ell \longrightarrow [0,T]$ such that $\tau^{\Lh^n+\frac{\cdot}{3T}} = \varphi_n(X_0, W, B)$. 

        Then by Assumption \ref{ass:coefficient_reward}, there exists a unique solution $X^{\tau_m}$ of the McKean-Vlasov SDE
        \begin{equation*}
            \begin{split}
            X^{\tau_m}_t
            ~ = ~
            X_0 + \int_{0}^{t \wedge \tau_m}b(s, X^{\tau_m}_{s\wedge \cdot}, \mu^{\tau_m}_s) dt
            & ~ + \int_{0}^{t \wedge \tau_m}\sigma(s, X^{\tau_m}_{s\wedge \cdot}, \mu^{\tau_m}_s) dW_s
            \\ & ~ + \int_{0}^{t \wedge \tau_m}\sigma_0(s, X^{\tau_m}_{s\wedge \cdot}, \mu^{\tau_m}_s) dB_s,
            ~
            t \in [0,T],
            ~
            \P-\mbox{a.s.}
            \end{split}
        \end{equation*}
        where for each $s \in [0,T]$, $\mu^{\tau_m}_s := \Lc^{\P}((X^{\tau_m}_{s\wedge \cdot}, W, \tau_m \wedge s)|B)$.
        By a similar argument in Lemma \ref{lemma:equivalence_approximation}, it follows that
        \begin{equation}\label{eq:approximation_tau_prop}
            \lim_{m \to \infty}
            \E^{\P}\bigg[\sup_{t \in [0,T]}|X^{\tau_m}_t - X_t|^p + \Wc_p(\mu^{\tau_m}_T,\mu_T)\bigg]
            ~ = ~ 0.
        \end{equation}
        Then it is straightforward to verify that $\P_m := \Lc^\P (X^{\tau_m}, W, B, \mu^{\tau_m}_T, \tau_m) \in \Pc_S(\nu)$.
        By Proposition \ref{prop:propagation_of_chaos}, for each $m \in \N_+$, it holds that
        \begin{equation}\label{eq:approximation_P_m_prop}
            \lim_{N \to \infty}
            \E^{\P^0}\bigg[
                \Wc_p\Big(\frac{1}{N}\sum_{i = 1}^{N}\delta_{(X^{i, N, \taub^{N, m}}, W^i, \tau^i_m)}, \Lc^{\P_m}((X, W,  \tau)|B)\Big)
            \bigg]
            ~ = ~
            0,
        \end{equation}
        where $\tau^i_m := \varphi_m(X^i_0, W^i, B)$, $\taub^{N, m} = (\tau^1_m, \cdots, \tau^N_m)$
        and $X^{N, \taub^{N, m}} = (X^{1, N, \taub^{N, m}}, \cdots, X^{N, N, \taub^{N, m}})$ is the unique strong solution to the large population stopped SDE given by \eqref{eq:N_SDE} with $\taub^{N, m}$.
        For simplicity, we write $\mu^{N, \taub^{N,m}}_T$ for $\frac{1}{N}\sum_{i = 1}^{N}\delta_{(X^{i, N, \taub^{N, m}}, \tau^i_m)}$.
        Therefore,
        \begin{align*}
             ~
             J(\P)
             ~ = & ~
                \E^{\P}\bigg[
                \int_{0}^{t \wedge \tau}f(s,X_{s\wedge \cdot}, \mu_s) ds
                + g(\tau, X_{\cdot}, \mu_T)\bigg]
            \\ ~ = & ~
                \E^{\P}\bigg[
                \int_{0}^{T}
                    \langle f(s, x, \mu_s)\mathds{1}_{[0,\theta)}(s), \mu_s(dx, dw, d\theta) \rangle
                    ds
                    +
                \langle g(\cdot, \mu_T), \mu_T \rangle
                \bigg]
            \\ ~ \le & ~
                \varliminf_{m \to \infty}
                \E^{\P_m}\bigg[
                \int_{0}^{T}
                    \langle f(s, x, \mu_s)\mathds{1}_{[0,\theta]}(s), \mu_s(dx, dw, d\theta) \rangle
                    ds
                    +
                \langle g(\cdot, \mu_T), \mu_T \rangle
                \bigg]
            \\ ~ \le & ~
                \varliminf_{N \to \infty}
                \varliminf_{m \to \infty}
                \E^{\P^0}\bigg[
                \int_{0}^{T}
                    \langle f(s, x, \mu^{N, \taub^{N,m}}_s)\mathds{1}_{[0,\theta]}(s), \mu^{N, \taub^{N,m}}_s(dx, dw, d\theta) \rangle
                    ds
                    \\
                    & ~~~~~~~~~~~~~~~~~~~~~~~~~~~~~~~~~~~~~~~~~~~~~~~~~~~~~~~~+
                \langle g(\cdot, \mu^{N, \taub^{N,m}}_T), \mu^{N, \taub^{N,m}}_T \rangle
                \bigg]
            \\ ~ = & ~
                \varliminf_{N \to \infty}
                \varliminf_{m \to \infty}
                \frac{1}{N}\sum_{i = 1}^{N}
                \E^{\P^0}\bigg[
                \int_{0}^{\tau^i_m}
                    f(s, X^{i,N,\taub^{N,m}}_{s\wedge \cdot}, \mu^{N, \taub^{N,m}}_s)\mathds{1}_{[0,\tau^i_m]}(s)
                    ds
                    \\
                    & ~~~~~~~~~~~~~~~~~~~~~~~~~~~~~~~~~~~~~~~~~~~~~~~~~~~~~+
                g(\tau^i_m, X^{i,N,\taub^{N,m}}, \mu^{N, \taub^{N,m}}_T)
                \bigg]
            \\ ~ \le & ~
                \varliminf_{N \to \infty}V_S^N(\nu_1, \cdots, \nu_N),
        \end{align*}
        where the first inequality holds by Assumption \ref{ass:coefficient_reward}, \eqref{eq:approximation_tau_prop}, Portmanteau Theorem (see Theorem 3.1 in \cite{FittePR})
        and the second one by \eqref{eq:approximation_P_m_prop}, Portmanteau Theorem (see Theorem 3.1 in \cite{FittePR}).
        Then we obtain the first inequality by arbitrariness of $\P$.

        The second inequality holds true similarly if we consider the following McKean-Vlasov SDE:
        \begin{align*}
            \Xh^{\tau}_t
            ~ = ~
            \Xh_0 + \int_{0}^{t \wedge \tau_m}b(s, \Xh^{\tau}_{s\wedge \cdot}, \muh^{\tau}_s) dt
            & ~ + \int_{0}^{t \wedge \tau_m}\sigma(s, \Xh^{\tau}_{s\wedge \cdot}, \muh^{\tau}_s) dW_s
            \\ & ~ + \int_{0}^{t \wedge \tau_m}\sigma_0(s, \Xh^{\tau}_{s\wedge \cdot}, \muh^{\tau}_s) dB_s,
            ~
            t \in [0,T],
            ~
            \P-\mbox{a.s.}
        \end{align*}
        where $\Xh_0$ is a $\R^n$-valued random variable indepedent of $(W,B)$ such that $\Lc(\Xh_0) = \frac{1}{N}\sum_{i = 1}^{N}\nu_i$, 
        and for each $s \in [0,T]$, $\muh^{\tau}_s := \Lc^{\P}((\Xh^{\tau}_{s\wedge \cdot}, W, \tau \wedge s)|B)$.
    \end{proof}

\subsection{Limit behavior of stopped rules with finite population}

In the subsection, we prove the relative compactness of finite population stopped rules and that the limit of its any convergent subsequence lies in $\Pc_W(\nu)$.

\begin{proposition}\label{prop:weak_main_theorem}
    Let Assumption \ref{ass:coefficient_reward} holds true, 
    and $\{\nu_i\}_{i \in \N_+} \subset \Pc_p(\R^n)$ be such that for some $\eps > 0$,
    $$
        \sup_{N \ge 1}\frac{1}{N}\sum_{i = 1}^{N}\int_{\R^n}|x|^{p + \eps}\nu_i(dx)
        ~ < ~
        +\infty,
    $$
    and $\{\P_N\}_{N \in \N_+}$ be given as in \eqref{eq:P_N_construct}.

    Then it holds that $\{\P_N\}_{N \in \N_+}$ is relatively compact under $\Wc_p$ and the limit of any convergent subsequence $\{\P_{N_m}\}_{m \in \N_+}$ of $\{\P_N\}_{N \in \N_+}$ belongs to $\Pc_W(\nu)$ for some $\nu \in \Pc_p(\R^n)$ with
    \begin{equation*}
        \lim_{m \to \infty}\frac{1}{N_m}\sum_{i = 1}^{N_m}\nu_i ~ = ~\nu.
    \end{equation*}
\end{proposition}

\begin{remark}
    The difference between Proposition \ref{prop:weak_main_theorem} and Theorem \ref{thm:main} lies in that the limit of any convergent subsequence $\{\P_{N_m}\}_{m \in \N_+}$ of $\{\P_N\}_{N \in \N_+}$ belongs to $\Pc_W(\nu)$ instead of $\Pc^*_W(\nu)$.
    Thus, it is a weaker version of our main theorem.
\end{remark}

\begin{proof}
    In the first part of the proof, we focus on the relatively compactness of $\{\P_{N}\}_{N \in \N_+} \subset \Pc_p(\Om)$. We will prove its tightness at first, then its relatively compactness. Since the tightness of a family of joint measures is equivalent to the tightness of all their marginal measures, we will consider all the marginal measures of $\{\P_{N}\}_{N \in \N_+}$.

    Recall that 
    \begin{equation*}
        \P_N
        ~ := ~
        \frac{1}{N}\sum_{i = 1}^{N}
        \Lc^{\P^0}
        \Big(X^{i,N,\taub^N}, W^i, B, \frac{1}{N}\sum_{i = 1}^{N}\delta_{(X^{i,N,\taub^{N}}, W^i, \tau^{i,N})}, \tau^{i,N}\Big).
    \end{equation*}

    Clearly, $\{\frac{1}{N}\sum_{i = 1}^{N}\Lc^{\P^0}(W^i, B)\}_{N \in \N_+}$ and $\{\frac{1}{N}\sum_{i = 1}^{N}\Lc^{\P^0}(\tau^{i,N})\}_{N \in \N_+}$ is tight.

    For the tightness of $\{\frac{1}{N}\sum_{i = 1}^{N}\Lc^{\P^0}(X^{i,N,\taub^N})\}_{N \in \N_+}$ in $\Pc_p(\Cc^n)$,
    one has the standard estimations, for some constant $C > 0$,
    \begin{align*}
        & ~
            \sup_{N \in \N_+}
                \P_N[\sup_{t\in [0,T]}|X_t| \ge a]
        \\ \le &   ~
            \sup_{N \in \N_+}
                \frac{1}{a^p}E^{\P_{N}}[\sup_{t\in [0,T]}|X_t|^p]
            =
            \sup_{N \in \N_+}
                \frac{1}{Na^p}\sum_{i = 1}^{N}\E^{\P^0}[\sup_{t\in [0,T]}|X^{i,N,\taub^N}_t|^p]
        \\ \le &   ~
            \sup_{N \in \N_+}
                \frac{1}{Na^p}\sum_{i = 1}^{N}C\E^{\P^0}[1 + |X^{i}_0|^p]
            =
            \frac{C}{a^p}\bigg(1 + \sup_{N \in \N_+}
                \frac{1}{N}\sum_{i = 1}^{N}\int_{\R^n}|x|^p\nu_i(dx)\bigg)
    \end{align*}
    and
    \begin{align*}
        ~
            \sup_{\tau \in \Tc^N}
                \P_N[|X_{(\tau + \delta)\wedge T} - X_\tau| \ge a]
        ~\le &  ~
            \sup_{\tau \in \Tc^N}
                \frac{1}{a^p}E^{\P_{N}}[|X_{(\tau + \delta)\wedge T} - X_\tau|^p]
        \\ = & ~
            \sup_{\tau \in \Tc^N}
                \frac{1}{Na^p}\sum_{i = 1}^{N}\E^{\P^0}[|X^{i,N,\taub^N}_{(\tau + \delta)\wedge T}- X^{i,N,\taub^N}_\tau|^p]
            ~ \le ~
            C\delta.
    \end{align*}
    Then by Aldous' criterion (see Theorem 16.10 in Billingsley \cite{Billingsley}),  $\{\frac{1}{N}\sum_{i = 1}^{N}\Lc^{\P^0}(X^{i,N,\taub^N})\}_{N \in \N_+}$ is tight in $\Pc_p(\Cc^n)$.

    \vspace{0.5em}

    For the tightness of $\{\Lc^{\P^0}(\frac{1}{N}\sum_{i = 1}^{N}\delta_{(X^{i,N,\taub^N},W^i, \tau^{i,N})})\}_{N \in \N_+}$
    in $\Pc_p(\Pc_p(\Cc^n \x \Cc^d \x [0,T]))$,
    we consider their mean measures $\{m\Lc^{\P^0}(\frac{1}{N}\sum_{i = 1}^{N}\delta_{(X^{i,N,\taub^N},W^i, \tau^{i,N})})\}_{N \in \N_+} \subset \Pc_p(\Cc^n \x \Cc^d \x [0,T])$,
    where for each $\P \in \Pc_p(\Pc_p(\Cc^n \x \Cc^d \x [0,T]))$,
    its mean measure $m\P \in \Pc_p(\Cc^n \x \Cc^d \x [0,T])$ is defined by
    \begin{equation*}
        m\P(A)
        ~ := ~
        \int_{\Pc_p(\Cc^n \x \Cc^d \x [0,T])}
            \mu(A)
        \P(d\mu),
        ~\mbox{for all}~
        A \in \Bc(\Cc^n \x \Cc^d \x [0,T]).
    \end{equation*}
    Then the tightness of all the mean measures' marginal measures follows by exactly the same argument for the tightness of $\{\frac{1}{N}\sum_{i = 1}^{N}\Lc^{\P^0}(X^{i,N,\taub^N})\}_{N \in \N_+}$, $\{\frac{1}{N}\sum_{i = 1}^{N}\Lc^{\P^0}(W^i, B)\}_{N \in \N_+}$
    and $\{\frac{1}{N}\sum_{i = 1}^{N}\Lc^{\P^0}(\tau^{i,N})\}_{N \in \N_+}$.
%
    Then by (2.5) in Sznitman\cite{Sznitman}, $\{\Lc^{\P^0}(\frac{1}{N}\sum_{i = 1}^{N}\delta_{(X^{i,N,\taub^N},W^i, \tau^{i,N})}\}_{N \in \N_+}$ is tight.

    \vspace{0.5em}

    Now we have the tightness of $\{\P_N\}_{N \in \N_+}$, the relatively compactness of $\{\P_N\}_{N \in \N_+}$  in $\Pc_p(\Om)$ will immediately hold true if we have the uniformly integrability of $\{\P_N\}_{N \in \N_+}$ in the sense that
    \begin{equation}\label{eq:ui_P_N}
        \lim_{a \to \infty}
            \sup_{N \in \N_+}
            \E^{\P_N}[|(X,W,B,\mu,\tau)|^p
            \mathds{1}_{\{|(X,W,B,\mu,\tau)| \ge a\}}]
        ~ = ~ 0,
    \end{equation}
    where for the product space $\Om$, we use $1$-product metric.
    In fact, one has the estimation
    \begin{align*}
        & ~
            \sup_{N \in \N_+}
            \E^{\P_N}\big[|(X,W,B,\mu,\tau)|^p
            \mathds{1}_{\{|(X,W,B,\mu,\tau)| \ge a\}}\big]
        \\ \le & ~
            \sup_{N \in \N_+}
            \frac{1}{a^\eps}\E^{\P_N}\big[|(X,W,B,\mu,\tau)|^{p +\eps}\big]
        \\ \le & ~
            \sup_{N \in \N_+}
            \frac{C}{Na^\eps}\sum_{i = 1}^{N}
            \E^{\P^0}\Big[\|X^{i,N,\taub^N}\|^{p +\eps} + \|W^i\|^{p +\eps} + \|B\|^{p +\eps}
            + \Wc^p_p(\mu^{N,\taub^N}_T,\delta_{(0,0,0)})
            + |\tau^{i,N}|^{p +\eps}\Big]
        \\ \le & ~
            \sup_{N \in \N_+}
            \frac{C}{Na^\eps}\sum_{i = 1}^{N}
            \E^{\P^0}\Big[\|X^{i,N,\taub^N}\|^{p +\eps} + \|W^i\|^{p +\eps} + \|B\|^{p +\eps}
            + |\tau^{i,N}|^{p +\eps}\Big]
        \\ \le & ~
            \frac{C}{a^\eps}\bigg(1 + \sup_{N \in \N_+}
                \frac{1}{N}\sum_{i = 1}^{N}\int_{\R^n}|x|^{p +\eps}\nu_i(dx)\bigg).
    \end{align*}
    Then \eqref{eq:ui_P_N} holds true by taking $a$ to $+\infty$.

    \vspace{0.5em}

    In the second part of the proof, we focus on the limit of any convergent subsequence of $\{\P_N\}_{N \in \N_+}$ and WLOG assume that $\{\P_N\}_{N \in \N_+}$ itself is convergent, i.e. there exists some $\P_{\infty} \in \Pc_p(\Cc^n \x \Cc^d \x \Cc^\ell \x \Pc_p(\Cc^n \x \Cc^d \x [0,T]) \x [0,T])$ such that
    \begin{equation*}
        \lim_{N \to \infty}\Wc_p(\P_N,\P_{\infty}) ~ = ~ 0,
    \end{equation*}
    and hence
    \begin{equation*}
        \lim_{N \to \infty}
        \Wc_p\bigg(
            \frac{1}{N}\sum_{i = 1}^{N}\nu_i,\Lc^{\P_\infty}(X_0)
        \bigg)
        ~ = ~
        0.
    \end{equation*}
    Then we verify that $\P_\infty \in \Pc_W(\nu)$, where hereinafter we denote $\Lc^{\P_\infty}(X_0)$ by $\nu$.

    \vspace{0.5em}

    \noindent \textbf{Proof of $(i)$:}
    First we observe that
    $\Lc^{\P_{\infty}}(B, W) = \lim_{N \to \infty}\Lc^{\P_{N}}(B, W) = \Wc^{d + \ell}$,
    where $\Wc^{d + \ell}$ denotes the $(d+\ell)$-dimensional Wiener measure,
    and next show that $(B, W)$ is a $(\F,\P_{\infty})$-Brownian motion.

    For any $\phi \in C_b(\Om)$ and $\varphi \in C_b(\Cc^d \x \Cc^\ell)$,
    it follows that for any $s, t \in [0,T]$ with $s < t$,
    \begin{align*}
        & ~
            \E^{\P_{\infty}}
                [\phi(X_{s\wedge\cdot}, W_{s\wedge\cdot}, B_{s\wedge\cdot}, \mu_s, \tau \wedge s)
                \varphi(W_{t\wedge \cdot} - W_{s\wedge\cdot}, B_{t\wedge \cdot} - B_{s\wedge\cdot})]
        \\ = & ~
            \lim_{N \to \infty}
            \frac{1}{N}\sum_{i = 1}^{N}\E^{\P^0}
                [\phi(X^{i,N,\taub^N}_{s\wedge\cdot}, W^i_{s\wedge\cdot}, B_{s\wedge\cdot}, \mu^N_s, \tau^{i,N} \wedge s)
                \varphi(W^i_{t\wedge \cdot} - W^i_{s\wedge\cdot}, B_{t\wedge \cdot} - B_{s\wedge\cdot})]
        \\ = & ~
            \lim_{N \to \infty}
            \frac{1}{N}\sum_{i = 1}^{N}\E^{\P^0}
                [\phi(X^{i,N,\taub^N}_{s\wedge\cdot}, W^i_{s\wedge\cdot}, B_{s\wedge\cdot}, \mu^N_s, \tau^{i,N} \wedge s)]
                \E^{\P^0}[\varphi(W^i_{t\wedge \cdot} - W^i_{s\wedge\cdot}, B_{t\wedge \cdot} - B_{s\wedge\cdot})]
        \\ = & ~
            \E^{\P_{\infty}}
                [\phi(X_{s\wedge\cdot}, W_{s\wedge\cdot}, B_{s\wedge\cdot}, \mu_s, \tau \wedge s)]
            \E^{\P_{\infty}}
                [\varphi(W_{t\wedge \cdot} - W_{s\wedge\cdot}, B_{t\wedge \cdot} - B_{s\wedge\cdot})].
    \end{align*}
    Then since the bounded continuous functions $\phi$, $\varphi$ are arbitrary, one has that $(W_{t\wedge \cdot} - W_{s\wedge\cdot}, B_{t\wedge \cdot} - B_{s\wedge\cdot})$ is independent of $\Fc_s$.

    \vspace{0.5em}

    \noindent \textbf{Proof of (ii):}
    It is straightforward to prove that $(B,\mu)$ is $\P_{\infty}$-independent of $(X_0,W)$.
    For any $\phi \in C_b(\Cc^\ell \x \Pc_p(\Cc^n \x \Cc^d \x [0,T]))$ and $\varphi \in C_b(\R^n \x \Cc^d)$, the law of large numbers, implies that
    \begin{align*}
        & ~
            \E^{\P_{\infty}}
                [\phi(B,\mu)\varphi(X_0,W)]
            -
            \E^{\P_{\infty}}[\phi(B,\mu)]
            \E^{\P_{\infty}}[\varphi(X_0,W)]
        \\ = & ~
            \lim_{N \to \infty}
                \frac{1}{N}\sum_{i = 1}^{N}
                \E^{\P^0}
                    [\phi(B,\mu^N_T)\varphi(X^i_0,W^i)]
                -
                \E^{\P^0}[\phi(B,\mu^N_T)]
                \frac{1}{N}\sum_{i = 1}^{N}\E^{\P^0}[\varphi(X^i_0,W^i)]
        \\ = & ~
            \lim_{N \to \infty}
                \E^{\P^0}
                    \bigg[\phi(B,\mu^N_T)
                    \bigg(\frac{1}{N}\sum_{i = 1}^{N}\varphi(X^i_0,W^i) - \frac{1}{N}\sum_{i = 1}^{N}\E^{\P^0}[\varphi(X^i_0,W^i)]\bigg)\bigg]
        \\ = & ~
            0.
    \end{align*}
    Then since the bounded continuous functions $\phi$, $\varphi$ are arbitrary, one has that $(B,\mu)$ is $\P_{\infty}$-independent of $(X_0,W)$.

    \noindent \textbf{Proof of (iii):}
    For the proof of the McKean-Vlasov SDE, the main tool we use is the martingale problem approach.
    By Theorem 4.5.2 in Stroock and Varadhan\cite{StroockVaradhan}, it is sufficient to prove that under $\P_{\infty}$,
    for any $\varphi \in C_b(\R^n \x \R^d \x \R^\ell)$, the process $(M^\varphi_t)_{t \in [0,T]}$ is a $(\F,\P_{\infty})$-martingale,
    where for each $t \in [0,T]$,
    \begin{equation*}
        M^\varphi_t
        ~ := ~
        \varphi(X_t,W_t,B_t)
        -
        \int_{0}^{t}\Lc_s\varphi(X,W,B,\mu,\tau)ds,
    \end{equation*}
    for all $(x,w,b,\theta,m) \in \Om$, $I^\theta_t := \mathds{1}_{[0,\theta]}(t)$,
    \begin{equation*}
        \Lc_t\varphi(x,w,b,m,\theta)
        ~ := ~
        \sum_{i = 1}^{n}
            b_i(t,x,m)I^\theta_t\partial_i\varphi(x_t,w_t,b_t)
        +
        \frac{1}{2}\sum_{i,j = 1}^{n + d + \ell}
            a_{i,j}(t, x, m)I^\theta_t\partial^2_{i,j}\varphi(x_t,w_t,b_t),
    \end{equation*}
    and
    \begin{equation*}
        a_{i,j}(t,x,m)
        ~ := ~
        \begin{bmatrix}
        \sigma & \sigma_0 \\
        I_{d \x d} & 0_{d \x \ell} \\
        0_{\ell \x d} & I_{\ell \x \ell}
        \end{bmatrix}
        \begin{bmatrix}
        \sigma & \sigma_0 \\
        I_{d \x d} & 0_{d \x \ell} \\
        0_{\ell \x d} & I_{\ell \x \ell}
        \end{bmatrix}^T
        (t,x,m).
    \end{equation*}
    Now for any $\phi_s \in C_b(\Om)$, $s, t \in [0,T]$ with $s < t$, we denote $(X^{i,N,\taub^N}_{s\wedge\cdot}, W^i_{s\wedge\cdot}, B_{s\wedge\cdot}, \mu^N_s, \tau^{i,N} \wedge s)$ by $\om^{i,N}_{s \wedge \cdot}$,
    and it follows that
    \begin{align*}
        & ~
            \E^{\P_{\infty}}
            [(M^\varphi_t - M^\varphi_s)\phi(\om_{s \wedge \cdot})]
        \\ = & ~
            \lim_{N \to \infty}
            \frac{1}{N}\sum_{i = 1}^{N}
                \E^{\P^0}
                \Bigg[\int_{s}^{t}
                \bigg(
                \sum_{j = 1}^{d}
                \partial_{j+n}\varphi(X^{i,N,\taub^N}_t,W^i_t,B_t)
                \Big(\sigma(r,X^{i,N,\taub^N}_{r\wedge \cdot},\mu^N_r) dW^i_r\Big)_j
            \\ & ~~~~~~~~~~~~~~~~~~~~~~~~~
                + \sum_{j = 1}^{\ell}
                \partial_{j+n+d}\varphi(X^{i,N,\taub^N}_t,W^i_t,B_t)
                \Big(\sigma_0(r,X^{i,N,\taub^N}_{r\wedge \cdot},\mu^N_r) dB_r\Big)_j
                \bigg)
                \phi(\om^{i,N}_{s \wedge \cdot})\Bigg]
        \\ = & ~
            0.
    \end{align*}

    \noindent \textbf{Proof of (iv):}
    The proof of the consistency condition \eqref{eq:consistency_weak} for $\mu_T$ is also straightforward.
    For any $\phi \in C_b(\Cc^\ell \x \Pc_p(\Cc^n \x \Cc^d \x [0,T]))$,
    $\varphi \in C_b(\Cc^n \x [0,T])$,
    it follows that, for each $t \in [0,T]$,
    \begin{align*}
        & ~
            \E^{\P_\infty}
            [\phi(B_{t \wedge \cdot},\mu_t)\varphi(X_{t\wedge\cdot}, W, \tau \wedge t)]
        \\ = & ~
            \lim_{N \to \infty}\frac{1}{N}\sum_{i = 1}^{N}
            \E^{\P^0}
            [\phi(B_{t \wedge \cdot},\mu^N_t)\varphi(X^{i,N,\taub^N}_{t \wedge \cdot}, W^i, \tau^{i,N} \wedge t)]
        \\ = & ~
            \lim_{N \to \infty}
            \E^{\P^0}
            \bigg[\phi(B_{t \wedge \cdot},\mu^N_t)\int_{\Cc^n  \x [0,T]}\varphi(x)\mu^N_t(dx)\bigg]
        \\ = & ~
            \E^{\P_\infty}
            \bigg[\phi(B_{t \wedge \cdot},\mu_t)\int_{\Cc^n  \x [0,T]}\varphi(x)\mu_t(dx)\bigg].
    \end{align*}
    Then since the functions $\phi$ and $\varphi$ are arbitrary, \eqref{eq:consistency_weak} holds true.

    \noindent \textbf{Proof of (v):}
    The following argument, wchich shows the $(H)$-Hypothesis type condition \eqref{eq:H_hypothesis} holds true for $\P^\infty$, works for $\Lc^{\P}(X, W, \tau)$, but fails for $\Lc^\P(X, \tau)$. 
    This is the reason that we use $\Lc^{\P}(X, W, \tau)$ instead of $\Lc^\P(X, \tau)$ throughout the paper.
    For and $\phi \in C_b(\Cc^n  \x C_b(\Cc^d) \x [0,T])$ and $\varphi_1, \varphi_2 \in C_b(\Cc^\ell \x \Pc_p(\Cc^n \x \Cc^d \x [0,T]))$,
    \begin{align*}
        & ~
            \E^{\P_\infty}[
            \E^{\P_\infty}[\phi(X_{t \wedge \cdot}, W, \tau \wedge t)
            \varphi_1(B_{t \wedge \cdot}, \mu_t)|\Gc_T]
            \varphi_2(B, \mu_T)
            ]
        \\ = & ~
            \lim_{N \to \infty}
            \E^{\P^0}\bigg[
            \frac{1}{N}\sum_{i = 1}^{N}
            \phi(X^{i,N,\taub^N}_{t \wedge \cdot}, W^i, \tau^{i,N} \wedge t)
            \varphi_1(B_{t \wedge \cdot}, \mu^N_t)
            \varphi_2(B, \mu^N_T)
            \bigg]
        \\ = & ~
            \lim_{N \to \infty}
            \E^{\P^0}\bigg[
            \int_{\Cc^{n} \x \Cc^d \x [0,T]}
                \phi(x, w, \theta)
                \mu^N_t(dx, dw, d\theta)
            \varphi_1(B_{t \wedge \cdot}, \mu^N_t)
            \varphi_2(B, \mu^N_T)
            \bigg]
        \\ = & ~
            \E^{\P_\infty}\bigg[
            \int_{\Cc^{n} \x \Cc^d \x [0,T]}
                \phi(x, w, \theta)
                \mu_t(dx, dw, d\theta)
            \varphi_1(B_{t \wedge \cdot}, \mu_t)
            \varphi_2(B, \mu_T)
            \bigg]
        \\ = & ~
            \E^{\P_\infty}\bigg[
            \int_{\Cc^{n} \x \Cc^d \x [0,T]}
                \phi(x, w, \theta)
                \mu_t(dx, dw, d\theta)
            \varphi_1(B_{t \wedge \cdot}, \mu_t)
            \E^{\P_\infty}[\varphi_2(B, \mu_T)|\Gc_t]
            \bigg]
        \\ = & ~
            \E^{\P_{\infty}}[
            \phi(X_{t\wedge \cdot}, W, \tau \wedge t)
            \varphi_1(B_{t \wedge \cdot}, \mu_t)
            \E^{\P_\infty}[\varphi_2(B, \mu_T)|\Gc_t]
            ]
        \\ = & ~
            \E^{\P_{\infty}}[
            \E^{\P_{\infty}}[
            \phi(X_{t\wedge \cdot}, W, \tau \wedge t)
            \varphi_1(B_{t \wedge \cdot}, \mu_t)
            |\Gc_t]
            \varphi_2(B, \mu_T)
            ].
    \end{align*}
    Then since $\phi$, $\varphi_1$, $\varphi_2$ are arbitrary,
    we can conclude the proof.
\end{proof}

\subsection{Proof of main theorems}

\proof \textbf{of Theorem \ref{thm:equivalence_value_function}}
    From the definition of $V_S(\nu)$ and $V_W(\nu)$, it is obvious to deduce that
    \begin{equation*}
        V_S(\nu) \le V_W(\nu).
    \end{equation*}
    On the other hand, for any $\P \in \Pc_W(\nu)$, and $\ell > 0$, by Lemma \ref{lemma:equivalence_approximation} and Lemma \ref{lemma:equivalence_strong_weak_discrete}, there exists a sequence of $\F^{X_0, W, B}$-stopping time $\{\taut^m\}_{m \in \N_+}$ together with a sequence of partitions $\{\pi_m: (t_i^m)_{i = 0}^m\}_{m \in \N_+}$ of $[0,T]$ with $0 = t_0^m < t_1^m < \cdots < t_m^m = T$ such that
    \begin{equation*}
        \lim_{m \to \infty}
        \Wc_p(\Lc^{\P}(\Xt^m, W^m, B^m, \mut^m_T, \taut^m),
        \Lc^{\P}(X, W, B, \mu_T, \tau))
        ~ = ~
        0,
    \end{equation*}
    where $\Xt^m$ is the unique solution of the following McKean-Vlasov SDE
    \begin{equation*}
        \begin{split}
        \Xt^m_t
        ~ = ~
        X_0 + \int_{t^m_1}^{\taut^m \wedge t \vee t^m_1}b(s, \Xt^m_{s\wedge \cdot}, \mut^m_s) dt
        & ~ + \int_{t^m_1}^{\taut^m \wedge t \vee t^m_1}\sigma(s, \Xt^m_{s\wedge \cdot}, \mut^m_s) dW^m_s
        \\ & ~ + \int_{t^m_1}^{\taut^m \wedge t \vee t^m_1}\sigma_0(s, \Xt_{s\wedge \cdot}, \mut^m_s) dB^m_s,
        ~
        t \in [0,T],
        \end{split}
    \end{equation*}
    with for each $t \in [0,T]$, $\mut^m_t$ is defined by
    \begin{equation*}
        \mut^m_t
        ~ := ~
        \Lc^\P((\Xt^m_{t \wedge \cdot}, W^m, \taut^m \wedge t)|B),
        ~
        t\in [0,T].
    \end{equation*}
    Then we define another sequence of processes $\{\Xh^m\}_{m \in \N_+}$ as the unique strong solution of the following McKean-Vlasov SDE
    \begin{equation*}
        \begin{split}
        \Xh^m_t
        ~ = ~
        X_0 + \int_{0}^{t \wedge \taut^m}b(s, \Xh^m_{s\wedge \cdot}, \muh^m_s) dt
        & ~ + \int_{0}^{t \wedge \taut^m}\sigma(s, \Xh^m_{s\wedge \cdot}, \muh^m_s) dW_s
        \\ & ~ + \int_{0}^{t \wedge \taut^m}\sigma_0(s, \Xh^m_{s\wedge \cdot}, \muh^m_s) dB_s,
        ~
        t \in [0,T],
        \end{split}
    \end{equation*}
    with for each $t \in [0,T]$, $\muh^m_t$ is defined by
    \begin{equation*}
        \muh^m_t
        ~ := ~
        \Lc^\P((\Xh^m_{t \wedge \cdot}, W, \taut^m \wedge t)|B),
        ~
        t\in [0,T].
    \end{equation*}
    It is standard to check that $\Lc^{\P}(\Xh^m, W, B, \muh^m_T, \taut^m)$ belongs to $\Pc_S(\nu)$.

    Then by almost the same argument as the proof of \eqref{eq:convergence_law_m} in Lemma \ref{lemma:equivalence_approximation},
    it follows that
    \begin{equation*}
        \lim_{m \to \infty}
        \Wc_p(\Lc^{\P}(\Xt^m, W^m, B^m, \mut^m_T, \taut^m),
        \Lc^{\P}(\Xh^m, W, B, \muh^m_T, \taut^m))
        ~ = ~
        0.
    \end{equation*}
    Therefore, by Assumption \ref{ass:coefficient_reward} that lower semi-continuity conditions of $f$, $g$, and Portmanteau Theorem (see Theorem 3.1 in \cite{FittePR}),
    it holds that
    \begin{equation*}
        J(\P)
        ~ \le ~
        \varliminf_{m \to \infty}
            J(\Lc^{\P}(\Xh^m, W, B, \muh^m_T, \taut^m))
        ~ \le ~
        V_S(\nu).
    \end{equation*}
    Since $\P \in \Pc_W(\nu)$ is arbitrary, one can deduce that $$V_W(\nu) \le V_S(\nu),$$
    and conclude the proof of $V_S(\nu) = V_W(\nu)$.

    When $\ell = 0$, $B$ vanishes, the argument is almost the same.
    The main difference lies in that the definitions of $\mut^m_t$ and $\muh^m_t$ for each $t \in [0,T]$, becomes
    \begin{equation*}
        \mut^m_t
        ~ := ~
        \Lc^\P((\Xt^m_{t \wedge \cdot}, W^m, \taut^m \wedge t)|U^m),
        ~
        \muh^m_t
        ~ := ~
        \Lc^\P((\Xh^m_{t \wedge \cdot}, W, \taut^m \wedge t)|U^m),
        ~
        t\in [0,T],
    \end{equation*}
    where $U^m$ is a $[0,1]^m$-valued uniformly distributed random variable indepedent of $(X_0, W, B)$.
    Moreover, the convergence result becomes
    \begin{equation*}
        \lim_{m \to \infty}
        \Wc_p\bigg(\int_{[0,1]^m}\Lc^{\P^m_u}(\Xt^m, W^m, B^m, \mut^m_T, \taut^m)du,
        \Lc^{\P}(\Xh^m, W, B, \muh^m_T, \taut^m)\bigg)
        ~ = ~
        0,
    \end{equation*}
    where $\{\P^m_u\}_{u \in [0,1]^m}$ is the conditional probability distributions of $\P$ given $U^m$.

    The convexity of $\Pc_W(\nu)$ is easy and standard to check.
    The closedness and relatively compactness of $\Pc_W(\nu)$ can be obtained by a very similar argument of Proposition \ref{prop:weak_main_theorem}.
\endproof

\vspace{0.5em}

\proof \textbf{of Theorem \ref{thm:main}:}
    \noindent \textbf{Proof of $(i)$:}
    For the sequence of empirical measures $\{\P_N\}_{N \in \N_+}$ given in the theorem,
    by Proposition \ref{prop:weak_main_theorem},
    it follows that
    $\{\P_N\}_{N \in \N_+}$ is relatively compact under $\Wc_p$
    and the limit $\P_\infty$ of any convergent subsequence $\{\P_{N_m}\}_{m \in \N_+}$ of $\{\P_N\}_{N \in \N_+}$ belongs to $\Pc_W(\nu)$ for some $\nu \in \Pc_p(\R^n)$ with
    \begin{equation*}
        \lim_{m \to \infty}\frac{1}{N_m}\sum_{i = 1}^{N_m}\nu_i ~ = ~\nu.
    \end{equation*}
    It remains to verify that $\P_{\infty} \in \Pc^*_W(\nu)$.
    By Proposition \ref{prop:limit_value_function}, it holds that
    \begin{align*}
        V_W(\nu)
         \ge 
        J(\P_{\infty})
         = 
        \lim_{m \to \infty}
        J(\P_{N_m})
         &  \ge 
        \varlimsup_{m \to \infty}
        V^{N_m}_S(\nu_1, \cdots, \nu_{N_m})
        \\ & \ge 
        \varliminf_{m \to \infty}
        V^{N_m}_S(\nu_1, \cdots, \nu_{N_m})
         \ge 
        V_S(\nu),
    \end{align*}
    therefore by Theorem \ref{thm:equivalence_value_function}, it follows that $J(\P_{\infty}) = V_W(\nu)$, i.e.
    $\P_{\infty} \in \Pc^*_W(\nu).$

    \vspace{0.5em}

    \noindent \textbf{Proof of $(ii)$:}
    For the proof of the first part, it holds directly from the Proof of $(i)$ that
    \begin{equation*}
        V_S(\nu)
        ~ = ~
        \lim_{N \to \infty}
        V^N_S(\nu_1, \cdots, \nu_N).
    \end{equation*}

    For the proof of the second part,
    let us denote by $\Pc_{S,c}(\nu)$ the collection of probability measures $\P$ in $\Pc_S(\nu)$ such that there exists some continuous function $\varphi: \R^n \x \Cc^d \x \Cc^\ell \longrightarrow [0,T]$, such that $\tau = \varphi(X_0, W, B)$ $\P$-a.s.

    Then by Proposition \ref{prop:propagation_of_chaos}, it holds that for any $\P \in \Pc_{S,c}(\nu)$, there exists a sequence of stopping times $\{\taub^{N} \in (\Tc^N)^N\}_{N \in \N_+}$ such that
    the corresponding sequence of probability measures $\{\P_N\}_{N \in \N_+}$, given by
    \begin{equation*}
        \P_N
        ~ := ~
        \frac{1}{N}\sum_{i = 1}^{N}
        \Lc^{\P^0}
        \Big(X^{i,N,\taub^N}, W^i, B, \frac{1}{N}\sum_{i = 1}^{N}\delta_{(X^{i,N,\taub^{N}}, W^i, \tau^i)}, \tau^i\Big),
    \end{equation*}
    constructed as \eqref{eq:P_N_construct} satisfies that
    \begin{equation*}
        \lim_{N \to \infty}\Wc_p(\P_N, \P) ~ = ~ 0.
    \end{equation*}

    In fact, the above approximation result also holds for any $\P \in \Pc_W(\nu)$, in particular $\P^* \in \Pc^*_W(\nu)$,
    if we notice that, by Theorem \ref{thm:equivalence_value_function}, $\Pc_S(\nu)$ is dense in $\Pc_W(\nu)$.
    and, by the argument in Proposition \ref{prop:limit_value_function},
    $\Pc_{S,c}(\nu)$ is dense in $\Pc_S(\nu)$.

    Then by Assumption \ref{ass:coefficient_reward}, Portmanteau Theorem (see Theorem 3.1 in \cite{FittePR}),
    it holds that
    \begin{equation*}
        \lim_{N \to \infty} J_N(\taub^N)
        ~ = ~
        \lim_{N \to \infty} J(\P_N)
        ~ = ~
        J(\P^*).
    \end{equation*}
    Finally, for each $N \in \N_+$, let us define
    $
        \eps_N
        ~ := ~
        V^N_S(\nu_1, \cdots, \nu_N) - J_N(\taub^N).
    $
    Then one can immediately deduce that
    \begin{equation*}
        \lim_{N \to \infty}
        \eps_N
        ~ = ~
        0.
    \end{equation*}

    \vspace{0.5em}

    \noindent \textbf{Proof of $(iii)$}
    WLOG, we assume that
    \begin{equation*}
        \varlimsup_{N \to \infty}
                \bigg|V^N_S(\nu_1, \cdots, \nu_N) - V_S\Big(\frac{1}{N}\sum_{i = 1}^{N}\nu_i\Big)\bigg|
        ~ = ~
        \lim_{N \to \infty}
                \bigg|V^N_S(\nu_1, \cdots, \nu_N) - V_S\Big(\frac{1}{N}\sum_{i = 1}^{N}\nu_i\Big)\bigg|,
    \end{equation*}
    and that
    \begin{equation*}
        \lim_{N \to \infty}\sum_{i = 1}^{N}\nu_i
        ~ = ~
        \nu,
        ~\mbox{for some}~
        \nu \in \Pc_p(\R^n).
    \end{equation*}
    Then by $(ii)$, it holds that
    \begin{align*}
        & ~
            \lim_{N \to \infty}
                \bigg|V^N_S(\nu_1, \cdots, \nu_N) - V_S\Big(\frac{1}{N}\sum_{i = 1}^{N}\nu_i\Big)\bigg|
        \\ \le & ~
            \lim_{N \to \infty}
                \big|V^N_S(\nu_1, \cdots, \nu_N) - V_S(\nu)\big|
            + \lim_{N \to \infty}
                \bigg|V_S(\nu) - V_S\Big(\frac{1}{N}\sum_{i = 1}^{N}\nu_i\Big)\bigg|
        \\ = & ~
            0.
    \end{align*}
    It remains to verify that
    \begin{equation*}
        \lim_{N \to \infty}
                \bigg|V_S(\nu) - V_S\Big(\frac{1}{N}\sum_{i = 1}^{N}\nu_i\Big)\bigg|
        ~ = ~ 0.
    \end{equation*}
    Let us consider a sequence of probability measures $\{\P^N_W\}$ with $\P^N_W \in \Pc_W\big(\frac{1}{N}\sum_{i = 1}^{N}\nu_i\big)$ 
    such that
    \begin{equation*}
        V_W\bigg(\frac{1}{N}\sum_{i = 1}^{N}\nu_i\bigg) \le J(\P^N_W) + \eps_N,
        ~\mbox{for all}~
        N \in \N_+.
    \end{equation*}
    where $\{\eps_N\}_{N \in \N_+}$ being a sequence of strictly positive real numbers with
    $
        \lim_{N \to \infty}\eps_N = 0.
    $
    Then the relatively compactness of $\{\P^N_W\}$ under $\Wc_p$ follows by a very similar argument of the relatively compactness of $\{\P_N\}_{N \in \N_+}$ in Proposition \ref{prop:weak_main_theorem}.
    Therefore, WLOG we assume that
    \begin{equation*}
        \lim_{N \to \infty}\Wc_p(\P^N_W, \P^\infty_W)
        ~ = ~ 0,
        ~\mbox{for some}~
        \P^\infty_W \in \Pc_W(\nu).
    \end{equation*}
    Then it holds that
    \begin{equation*}
        V_W(\nu)
        \ge
        J(\P^\infty_W)
        =
        \lim_{N \to \infty}
        J(\P^N_W)
        \ge
        \varliminf_{N \to \infty}
        V_W\bigg(\frac{1}{N}\sum_{i = 1}^{N}\nu_i\bigg)
        =
        \varliminf_{N \to \infty}
        V_S\bigg(\frac{1}{N}\sum_{i = 1}^{N}\nu_i\bigg)
        \ge
        V_S(\nu),
    \end{equation*}
    where the last inequality holds by Proposition \ref{prop:limit_value_function}.
\endproof





\end{document}